\documentclass[11pt, reqno]{amsart}

\usepackage[utf8]{inputenc}

\usepackage{graphicx}
\usepackage{comment}

\usepackage{fourier}
\usepackage[margin=0.75in]{geometry}

\usepackage[T1]{fontenc}

\usepackage{amsthm}
\usepackage{amsmath}
\usepackage{amssymb}
\usepackage{mathtools}
\usepackage{apptools}
\usepackage{setspace}
\usepackage{url}
\usepackage{tabularx}

\usepackage[
backend=biber,
style=numeric,
sorting=nyt,
maxnames=10,
giveninits=true
]{biblatex}
\newcommand {\R}{\mathbb{R}}

\newcommand {\Z}{\mathbb{Z}}

\newcommand {\C}{\mathbb{C}}

\newcommand{\la}{\lambda}
\newcommand{\re}{\operatorname{Re}}
\newcommand{\im}{\operatorname{Im}}
\newcommand{\sumast}{\mathop{\sum\nolimits^{\mathrlap{\ast}}}}

\newcommand{\Mod}[1]{\ (\mathrm{mod}\ #1)}
\newtheorem{thm}{Theorem}[section]

\newtheorem{lemma}[thm]{Lemma}
\newtheorem{prop}[thm]{Proposition}
\newtheorem{cor}[thm]{Corollary}
\newtheorem{remark}{Remark}
\newtheoremstyle{named}{}{}{\itshape}{}{\bfseries}{.}{.5em}{\thmnote{#3}}
\theoremstyle{named}
\newtheorem*{namedthm}{Theorem}

\AtAppendix{\counterwithin{lemma}{section}}
\numberwithin{equation}{section}

\usepackage{multicol}
\usepackage{xcolor}
\usepackage{hyperref}
\hypersetup{
    colorlinks,
    linkcolor={red!50!black},
    citecolor={blue!50!black},
    urlcolor={blue!80!black}
}

\usepackage{soul}

\hypersetup{breaklinks=true}

\usepackage[inline]{enumitem}

\DeclareMathOperator{\sgn}{sgn}

\title{Character sum, reciprocity and Voronoi formula}
\author{Chung-Hang Kwan}
\address{University College London}
\email{ucahckw@ucl.ac.uk}

\author{Wing Hong Leung}
\address{Rutgers University}
\email{joseph.leung@rutgers.edu}

\date{}

\begin{document}

\begin{abstract}
We prove a novel four-variable character sum identity which serves as a twisted, non-archimedean counterpart to Weber's integrals for Bessel functions. Using this identity and ideas from Venkatesh's thesis,  we present a new, spectral proof of the Voronoi formula for classical modular forms. 
\end{abstract}

\subjclass[]{11F11, 11F12, 11F66, 11F72, 11L05}
\keywords{Modular form, automorphic form, $L$-function,  functional equation, Voronoi formula,   Petersson trace formula, character sum, Kloosterman sum, Gauss sum,  spectral identity, Analytic Number Theory}

\maketitle

\section{Introduction}

\subsection{Character sums and modular forms}  



Character sums of various forms have been instrumental in the development of number theory. A notable example is Kloosterman's 1926 work on the circle method and representation of natural numbers by quadratic forms, where the following character sum, now known as the \emph{Kloosterman sum}, proved fundamental:
\begin{align}
    S(m,n;c)  := \sideset{}{^*}{\sum}\limits_{x(\bmod\, c)} \, e\left(\frac{mx+n \, \overline{x}}{c}\right), 
\end{align}
where  $m, n, c$ are integers and $c\ge 1$.  His method can be applied more generally to estimate the sizes of Fourier coefficients of holomorphic modular forms, providing improvements over Hecke's bound; see \cite[Chp. 20.3]{IK} and  \cite[Sect. 3.4]{BB13}. The relationship between the Kloosterman sum and modularity can also be understood through \emph{trace formulae} after Petersson and Selberg; see \cite[Chp. 14]{IK}.


Another intriguing connection was revealed by Kuznetsov \cite[Thm. 4 and Sect. 7.1]{Kuz} while developing his celebrated trace formula.  He proved the remarkable identity 
\begin{align}\label{Selb}
    S(m,n;c) =   \sum_{d\mid (m,n, c)}  S(1, mn/d^2; c/d)d,
\end{align}
by applying Hecke's multiplication rule
\begin{align}
    \lambda_{F}(m) \lambda_{F}(n)  =  \sum_{d\mid (m, n)}  \lambda_{F}(mn/d^2)
\end{align}
to his formula, where $\lambda_{F}(\,\cdot\,)$ denotes the normalized Hecke eigenvalues of a cuspidal eigenform $F$ of level $1$. Historically, (\ref{Selb})
was discovered by Selberg in the 1930s,  but no proof was provided. As noted in \cite[p. 2318]{BK-JFA}, a direct proof of (\ref{Selb}) has resisted straightforward attempts.  



Identities of character sums also serve as essential local inputs in the arithmetic and analytic aspects of modular forms (or more generally, automorphic forms) via trace formulae. For instance, Iwaniec \cite{Iw87} obtained an averaged form of Waldspurger/Kohnen--Zagier's formula through a delicate identity relating Kloosterman sums to Sali\'e sums, which can be considered as a non-archimedean analogue of the integral identities of Bessel functions due to Hardy and Weber (see \cite[Lem. 4]{Iw87},  \cite[Thm. 1.1, Cor. 8.4, Sects. 18 and 22]{BM05} and \cite{CQ20}). As another example, Duke--Iwaniec \cite{DI93} discovered a curious identity between the Kloosterman sum and a cubic exponential sum, which plays an important role in the cubic liftings of \cite{MR99, FO22}. This identity is a non-archimedean version of Nicholson's identity for the Bessel and Airy functions. 

In this work, we prove a new reciprocity identity of a four-variable character sum associated with a Dirichlet character $\chi\, (\bmod\, q)$.  This result is the intrinsic form of an identity of Kloosterman sums (Prop. \ref{cor:CharSumId}) which should be interpreted as a twisted, non-archimedean counterpart of another Weber's integral identity (eq. (\ref{Web})).

\begin{thm}\label{Thm:CharSum} Let $q, \ell \ge 1$, $a, b, u, v$ be integers such that $ab|q^\infty$ and $(uv,q)=1$. \footnote{ The notation $c\mid q^{\infty}$ refers to $c\mid q^{k}$ for some $k\ge 1$. } Let $\chi\, (\bmod\, q)$ be a Dirichlet character. Define
\begin{align}\label{keyacharsum}
    \mathcal{C}_\chi^\ell(a,u,b,v) \ := \ \chi(u) \sumast_{\substack{\alpha\, (\bmod\, a)\\ bv\equiv-\alpha q\, (\bmod\, a)}}e\left(\frac{\ell\overline{\alpha u}}{a}\right)\overline{\chi}\left(\frac{\alpha q+bv}{a}\right). 
\end{align}
Then we have
\begin{align}\label{IntroCharId1}
    \mathcal{C}_\chi^\ell(a,u,b,v) \ = \ e\left(-\frac{\ell q\overline{uv}}{ab}\right) \, \overline{\mathcal{C}_\chi^\ell(b,v,a,u)}.
\end{align}
\end{thm}

Our identity (\ref{IntroCharId1}) features the switching of parameters $a \leftrightarrow b$ and $u\leftrightarrow v$, despite their distinct roles in (\ref{keyacharsum}). Similar to the instances discussed earlier, (\ref{IntroCharId1}) has applications to modular forms, as discussed in the next section. 


\subsection{Global results}


There has been considerable interest in deriving the analytic continuation and functional equations of $L$-functions associated with modular forms \emph{directly} from their Dirichlet series, avoiding the traditional reliance on integral representations as in the cases of Hecke, Rankin and Selberg. In his seminal 1977 proceedings \cite[pp. 141--142]{Zag77}, Zagier suggested an approach based on the property that the set of holomorphic Poincar\'{e} series spans the linear space of holomorphic cusp forms. This idea was later implemented in two different cases by Goldfeld \cite{Go79} and Mizumoto \cite{Miz84}. The method has since found various arithmetic applications as explained in e.g.,  Goldfeld--Zhang \cite{GZ00}, Nelson \cite{Nel13}, and Diamantis--O'Sullivan \cite{DO10}.



In his thesis \cite[Chp. 3]{Ven-thesis}, Venkatesh gave an elegant new proof of the classical converse theorem for holomorphic modular forms of level $1$, which may be viewed as a limiting variant of \cite[Sect. 2]{Go79} (cf. \cite{BFL23}). Several aspects of his approach are distinctive both conceptually and technically.  Venkatesh circumvented the need for the aforementioned property of the holomorphic Poincar\'{e} series with a very neat argument. In fact, his method works equally well for non-holomorphic modular forms by utilizing the Kuznetsov trace formula in place of the Petersson trace formula which he employed in the holomorphic case. Another idea from \cite{Ven-thesis}  is to begin with Dirichlet polynomials instead of working directly with $L$-functions. Recent works of Booker--Farmer--Lee \cite{BFL23} and Blomer--Leung \cite{BL24} have yielded more general converse theorems using this approach.



In this article, we employ Venkatesh's ideas to provide new, spectral proofs of the functional equation for the twisted Hecke $L$-function and the Voronoi formula for holomorphic cusp forms \footnote{attributed to Wilton in \cite[Sect. 3]{MS04}. }, in line with the philosophy of ``beyond endoscopy'' (see \cite{Ven04}). We derive these results from a spectral identity for averages of smooth Dirichlet polynomials (Thm. \ref{ourmainthmtwist}), where Thm. \ref{Thm:CharSum} plays a significant role. Our work also uses the circle of techniques under  ``spectral reciprocity'' (e.g., \cite{BK-duke, BK-JFA}), but with several important variations as we will discuss in Sect. \ref{roadmap}.




Our main results are described as follows. Let $q\ge 1$ be an integer, $e(z):= e^{2\pi i z}$, $e_{q}(z):= e^{2\pi i z/q}$, and $\chi\, (\bmod\, q)$ be a primitive Dirichlet character. The Gauss sum associated with $\chi$ is defined by
\begin{align}\label{gaussum}
    \epsilon_\chi  :=  \frac{1}{\sqrt{q}}\sum_{\alpha\Mod{q}}\chi(\alpha)e_q(\alpha).
\end{align}
Suppose $\mathcal{B}_{k}(1)$ is an orthogonal basis of holomorphic cusp forms of level $1$ and weight $k$, for which the Fourier expansion of $f\in \mathcal{B}_{k}(1)$ is given by
\begin{align}
    f(z) \, &= \,  \sum_{n= 1}^{\infty} \, \lambda_{f}(n)n^{\frac{k-1}{2}} e(nz), 
\end{align}
where $\lambda_{f}(1)=1$,  $ z=x+iy$ with  $x,y\in\R$ and $y>0$. The harmonic weighting refers to
\begin{align*}
    \sideset{}{^h}{\sum}_{f \in \mathcal{B}_{k}(1)}  \alpha_{f} \,  := \,   \frac{(4\pi)^{k-1}}{ \Gamma(k-1)}  \sum_{f \in \mathcal{B}_{k}(1)}  \, \frac{\alpha_{f}}{||f||^{2}}. 
\end{align*}
For $\re s  \gg  1$, the twisted Hecke $L$-function associated with $f$ and $\chi$ is defined by the Dirichlet series
\begin{align}
	L(s,f\times\chi ) \, &:= \, \sum_{n= 1}^{\infty} \frac{\lambda_{f}(n)\chi(n)}{n^s}. \label{holoDS}
\end{align} 
Denote by  $J_{k-1}(\cdot)$  the $J$-Bessel function, and let
\begin{align}\label{gl2gammfac}
 	\gamma_{k}(s) \ := \  2^{(3-k)/2} \sqrt{\pi}\,  (2\pi)^{-s} \, \Gamma\left(s+ \frac{k-1}{2}\right) \ =  \ \pi^{-s}\,  \Gamma\left(\frac{s+\frac{k-1}{2}}{2}\right) \Gamma\left(\frac{s+\frac{k+1}{2}}{2}\right).
 \end{align}

\begin{thm}\label{ourmainthmtwist}
For any $\ell\geq1$ and $g\in C_{c}^{\infty}(0,\infty)$, we have
\begin{align}\label{mainglobid}
  \sideset{}{^h}{\sum}_{f \in \mathcal{B}_{k}(1)}  \lambda_{f}(\ell) \sum_{n=1}^{\infty} \,  \lambda_{f}(n)\chi(n) g(n) =
  2\pi i^{k}\frac{\epsilon_\chi^2}{q}\ \ \sideset{}{^h}{\sum}_{f \in \mathcal{B}_{k}(1)} \ \lambda_{f}(\ell)\sum_{n=1}^\infty\,  \lambda_{f}(n)\overline{\chi}(n)	\int_{0}^{\infty}   g(y)J_{k-1}\left(\frac{4\pi\sqrt{ny}}{q}\right)  dy. 
\end{align}    
\end{thm}


In Sect. \ref{proofmainisola},  we isolate a single cusp form from (\ref{mainglobid}), which requires a careful discussion on the \emph{analytic continuation} and \emph{polynomial growth} of $L(s, f \times \chi)$; see Sect. \ref{ACPG} and App. \ref{AC&PGsec}.  Hence, we have
\begin{cor}\label{maincor}
   Let $f\in \mathcal{B}_{k}(1)$ and $\chi \, (\bmod\, q)$ be a primitive Dirichlet character. Then for any $s\in \C$, 
   \begin{align}\label{stdfunceq}
            L(s, f\times\chi)   =  i^k \epsilon_\chi^2q^{1-2s}\,  \frac{\gamma_{k}(1-s)}{\gamma_{k}(s)}   L(1-s, f\times \overline{\chi}).
    \end{align}
\end{cor}

Our argument for Thm. \ref{ourmainthmtwist} applies to additive twists with very minor adjustments. 
\begin{cor}\label{addVor}
   Let $q\ge 1$, $(a,q)=1$, and $a\overline{a}\equiv 1 (\bmod\, q)$. Then for  $g\in C_{c}^{\infty}(0,\infty)$, we have
	\begin{align}\label{vorformst}
		\sum_{n=1}^{\infty} \lambda_{f}(n)e_{q}\left(an\right) g(n)  =   \frac{2\pi i^k}{q} \sum_{n=1}^{\infty} \lambda_{f}(n) e_{q}\left(-\overline{a} n\right) \int_{0}^{\infty} g(x) J_{k-1}\left(\frac{4\pi\sqrt{nx}}{q}\right)  dx.
	\end{align}
\end{cor}

\noindent  Alternatively, one may deduce Cor. \ref{addVor} from Cor. \ref{maincor} using \cite[Thm. 1.3]{kiral2016voronoi}. We refer the reader to \cite{Te19}, \cite{MS04}, \cite{AC21} for other recent approaches to proving the classical Voronoi formula (\ref{vorformst}).


\begin{remark}
    We are aware of an article by Herman \cite{HermanBeFe}. After a thorough review, we identified major gaps and inaccuracies in his approach, which we rectified in our work \cite{KLev}.  In addition, our work provided generalizations and new observations.  Readers are referred to \cite[Sect. 1.6]{KLev} and \cite{Erra} for a detailed comparison between the two works. 
\end{remark}


\begin{remark}
Generalizing our main results to cusp forms of level $D$ with nebentypus (say primitive) should not pose significant difficulty provided that $(D,q)=1$,  apart from the need of more involved notations. This can be achieved by combining the arguments for Thm. \ref{ourmainthmtwist} with that for \cite[Thm. 1.2]{KLev}. 

This work and \cite{KLev} are both complementary and
independent in interest, as they address different types of ramifications which require entirely different treatment of the character sums. Notably, Thm.  \ref{Thm:CharSum}  is unique to this paper, and the bulk of our argument, spanning Sect. \ref{sect:step2} to Sect.  \ref{Backwd}, is distinct from \cite{KLev}; see also our comments in Sect. \ref{localintro} and Sect. \ref{cantermglob}. 
\end{remark}

\begin{remark}
The manipulations of the character sums, which is the focus of this paper, carry over to the non-holomorphic case. In the non-holomorphic case, it suffices to apply the Kuznetsov trace formula and adapt the argument for Prop. \ref{prelimbd} to account for the different weight functions.

\end{remark}


\subsection{A road map for Theorem \ref{ourmainthmtwist}}\label{roadmap}

Although we initially represent the $L$-functions in terms of their Dirichlet series, the essential components of our argument are inherently local; see Thm. \ref{Thm:CharSum}, Prop. \ref{cor:CharSumId} and eq. (\ref{Web}). In Sects. \ref{Petersect}--\ref{sect:step2}, we apply the Petersson formula and Poisson summation to the left-hand side of (\ref{mainglobid}),  getting
\begin{align}\label{naigeoexp}
     g(\ell)\chi(\ell)  +  \frac{2\pi i^{-k}}{q} \, \sum_{c\ge 1} \, \frac{1}{c^2}   \sum_{m\in \Z} \, \widehat{\mathcal{D}}_\chi^\ell(m,c) \int_{0}^{\infty}  g(y)J_{k-1}\bigg(\frac{4\pi\sqrt{y\ell}}{c}\bigg) e_{cq}(-my)  dy,
\end{align}
 where  $\mathcal{D}_\chi^\ell(m,c) := \chi(m)S(m,\ell;c)$, and
\begin{align}\label{FTcharsum}
    \widehat{\mathcal{D}}_\chi^\ell(m,c)  :=  \sum_{\gamma \,(cq)} \, \mathcal{D}_\chi^\ell(\gamma,c)e_{cq}(m\gamma)  =  \sum_{\gamma\,(cq)}\chi(\alpha)S(\gamma,\ell;c)e_{cq}(m\gamma). 
\end{align}


\subsubsection{Two dualities}\label{geodual}

The adelic viewpoint suggests that $J_{k-1}$ serves as the archimedean counterpart of $S(m,\ell;n)$. Correspondingly, our argument hinges on two intriguing geometric dualities, one for the \emph{Fourier--Hankel transform} of $J_{k-1}$ and the other for the \emph{finite Fourier transform}  $\widehat{\mathcal{D}}_\chi^\ell(\ \cdot\ ,c)$. 


The first duality is commonly known as Weber's second exponential integral identity (Lem. \ref{keyspec}):
\begin{align}\label{Web}
    \int_{0}^{\infty} e(\alpha y) J_{k-1}(4\pi \beta\sqrt{y}) J_{k-1}(4\pi \gamma\sqrt{y}) dy  =  i e\left( \sgn(\alpha)\frac{k-1}{4}\right) \cdot  \frac{1}{2\pi \alpha}\, \cdot \,   J_{k-1}\left( \frac{4\pi\beta\gamma}{|\alpha|}\right) \cdot e\left( - \frac{\beta^2+\gamma^2}{\alpha}\right). 
\end{align}   
The \emph{Hankel inversion formula} (Lem. \ref{hankelinv}) is a limiting form of (\ref{Web}). In Sect. \ref{Hankana}, we will apply  Hankel's formula and Weber's identity successively. The second duality is a twisted, non-archimedean analogue of (\ref{Web}): 
\begin{prop}\label{cor:CharSumId}
   Let $\ell, m,c\in \Z$ with $m,c\ge 1$ and $\chi \, (\bmod\, q)$ be a primitive character. Then 
    \begin{align}\label{IntroCharId2}
        \widehat{\mathcal{D}}_\chi^\ell(m,c)  \, = \, (\epsilon_\chi)^2 \, \cdot \, \frac{c}{m} \, \cdot \,\overline{\widehat{\mathcal{D}}_\chi^\ell(c,m)} \, \cdot \,e\left(-\frac{1}{mc}\right).
    \end{align}
\end{prop}
\noindent Prop. \ref{cor:CharSumId}, whose proof relies on Thm. \ref{Thm:CharSum}, is non-trivial. A sketch will be provided in Section \ref{localintro}.


\subsubsection{Multiplicative (analytic--arithmetic) cancellation}\label{AAcanc}

The next step is to apply Poisson summation to the $c$-sum in (\ref{naigeoexp}), though an 
observant reader might question its applicability
due to the apparent singularity at $c=0$. However,  the factor $1/c^2$ of (\ref{naigeoexp}) is cancelled out upon inserting  (\ref{IntroCharId2}) and (\ref{Web}) (with appropriately chosen $\alpha, \beta, \gamma$)! Additionally,  the factor $e(-1/mc)$ from (\ref{IntroCharId2}) perfectly cancels the final exponential factor in (\ref{Web}), paving the way for the two upcoming steps.

This technical feature is also crucial in previous works on beyond endoscopy, where the Selberg trace formula was analyzed and the singularities from the archimedean orbital integrals are more subtle (see  \cite[Sect. 4]{Alt-1}, \cite{GKM+18},  \cite[Sect. 6]{emory2024BE}).  Once again, smoothing of these singularities is necessary for  Poisson summation, but this requires delicate use of an approximate functional equation and the class number formula.


\subsubsection{Three-fold reciprocity and local analysis}\label{localintro}

In (\ref{FTcharsum}) (cf. Sect. \ref{Recip}), we apply a \emph{first} additive reciprocity: 
\begin{align}\label{reciprel}
    \frac{\overline{m}}{c} \ + \ \frac{\overline{c}}{m} \ \equiv \ \frac{1}{mc} \ \ (\bmod\, 1). 
\end{align}
 In Sect. \ref{arithpre}, we split the $c$- and $m$-sums in (\ref{naigeoexp}) appropriately. One roughly gets:
\begin{align}\label{essentpart}
      & \frac{\epsilon_\chi}{q^{3/2}}\sum_{c_0|q^\infty} \, \sum_{m_{0} \mid q^{\infty}} \, \sum_{\substack{m'\ge 1\\ (m',q)=1}} \, \frac{1}{m_{0}m'}\sum_{\substack{c'\neq0\\ (c',m'q)=1}} \,  e_{c_0m_0m'}(\ell q\overline{c'}\, )\mathcal{C}_\chi^\ell(c_0,c',m_0,m') \, \times \, (\text{archimedean part}),
\end{align}
where $\mathcal{C}_{\chi}^{\ell}(\cdots)$ is the character sum defined in Thm. \ref{Thm:CharSum}.
See (\ref{roleofkeychar}) for the precise expressions.

The character sum (\ref{keyacharsum}) exhibits twisted multiplicativity  (Lem. \ref{lemma:charmult}), enabling a \emph{local} analysis. Thm. \ref{Thm:CharSum} follows from the method of 
$p$-adic stationary phase (\cite[Chp. 12.3]{IK}) to be carried out in Sect. \ref{sect:ArithPrep}, along with a \emph{second} use of (\ref{reciprel}); see (\ref{anotrecip}). Our proof of  (\ref{IntroCharId1}), and hence Thm. \ref{ourmainthmtwist}, is non-trivial even when $q$ is a prime.

A \emph{third} use of (\ref{reciprel}) allows us to combine the exponential phase of (\ref{IntroCharId1}) with that of (\ref{essentpart}), then a second copy of $\epsilon_{\chi}$ arises upon opening up $\overline{\mathcal{C}_\chi^\ell}(\cdots)$ by its definition (\ref{keyacharsum}) and gluing variables back together; see (\ref{anotherGauss}) of Sect. \ref{Backwd}.  Now, Prop. \ref{cor:CharSumId} follows from this. The factor $(\epsilon_{\chi})^2$ from (\ref{IntroCharId2}) will become the root number of (\ref{stdfunceq}).


\subsubsection{Additive cancellations and back to the global aspect}\label{cantermglob}

There are two subtle yet crucial cancellations of terms when proving Thm. \ref{ourmainthmtwist}. Namely, the diagonal term and dual zeroth frequency from the first applications of Petersson formula and Poisson summation cancel perfectly with the zeroth term (on the physical side) and diagonal term from the second applications of Poisson summation and Petersson formula, respectively. The first cancellation uses the Hankel inversion formula and the fact that $|\epsilon_{\chi}|=1$. 

Thm. \ref{ourmainthmtwist} follows from the observations above, the \emph{interchange} of roles between the $c$-sum and $m$-sum in (\ref{naigeoexp}), as well as the
\emph{backward} applications of the Poisson and Petersson formulae in Sect. \ref{Backwd}. 

In Sect. \ref{proofmainisola}, we deduce Cor. \ref{maincor} from Thm. \ref{ourmainthmtwist}. This requires analytic inputs from App. \ref{AC&PGsec} such that a series of contour shifts is justified. The associated subtleties will be discussed below in Sect. \ref{ACPG}.



\subsubsection{Remarks on analytic continuation}\label{ACPG}

Obtaining analytic continuation to $\re s> 1/2-\delta$, for some absolute $\delta>0$, is crucial as a prerequisite for discussing the functional equation. However, this is non-trivial for methods based on Dirichlet series. For example, using his refined analysis of the Eichler--Selberg trace formula  \cite{Alt-1, Alt-2}, Altu\u{g} established holomorphy for the standard
$L$-function of Ramanujan's $\Delta$-function only in  $\re s>31/32$; see \cite[Cor. 1.2]{Alt-3}. Other methods also fall short of reaching $\re s> 1/2-\delta$, as illustrated in \cite{WhiteBC}, \cite{GM20}, \cite[Prop. 10]{Ven-thesis}, \cite{Her11}, and \cite{DI94}.

In  \cite[Chp. 2.6]{Ven-thesis}, Venkatesh had the idea of embedding the holomorphic cusp forms of weight $k$ and level $1$ into the \emph{full spectrum} of $L^{2}$-automorphic forms of level $1$, consisting of 
Maass cusp forms of weight $0$, holomorphic cusp forms of \emph{all weights} and the Eisenstein series. Thus, an arbitrary smooth compactly test function on $(0,\infty)$ can be put on the \emph{geometric} side of the trace formula, i.e.,  using the \emph{arithmetic Kuznetsov (or Petersson--Kuznetsov)} trace formula  (see \cite[Thm. 16.5]{IK}). This simplifies the analysis, reducing it to handling a basic Fourier integral. For some \emph{spectral} test functions $H$,  one has
\begin{align}\label{limitfor}
    \sideset{}{^h}{\sum}_{\substack{f: \text{ full spectrum }\\ \text{of level $1$}}} \, H(t_{f})\lambda_{f}(\ell) \, \sum_{n} \, g(n/X) \lambda_{f}(n) \chi(n)  =  O_{A}(X^{-A})  \hspace{25pt} (X\to \infty). 
\end{align}
The goal is to deduce an entire continuation of $L(s, f\times \chi)$ from (\ref{limitfor}). 

However, the space of admissible spectral test functions $H$ is somewhat limited and the technical challenge lies in constructing a test function in this space that ``effectively'' isolates a specific part of the spectrum. This problem was addressed in \cite[Chp. 6.3]{Ven-thesis} by an elaborate analysis of the Sears--Titchmarsh transform. In  \cite[p. 50]{Ven-thesis}, an extra technical assumption concerning the growth in spectral parameters (Laplace eigenvalues/weights) is needed (cf. \cite[eq. (2.35)]{Ven-thesis} and \cite[Rmk. 3.17]{WhiteBC}).

In App. \ref{AC&PGsec}, we address this technicality by applying the 
Petersson and Poisson formulae, resulting in a more intricate oscillatory integral. We use the stationary calculus of \cite{BKY} to show that $L(s, f\times \chi)$ admits an analytic continuation and has polynomial growth on $\re s>-(k-6)/2$.


\subsection{Acknowledgement}
The research is supported by the EPSRC grant: EP/W009838/1.

\section{Preliminaries}\label{prelimsec}

Let $\phi \in C_{c}^{\infty}(\R)$ and\,  $\psi \in C_{c}^{\infty}(0,\infty)$. The \emph{Fourier transform} of $\phi$ and the \emph{Mellin transform} of $\psi$ are 
\begin{align*}
     \widehat{\phi}(y) := \int_{\R} \ \phi(x) e(-xy)  dx \hspace{10pt} (y  \in  \R) \hspace{10pt} \text{ and } \hspace{10pt} 
     \widetilde{\psi}(s) :=   \int_{0}^{\infty}  \psi(x) x^{s-1}  dx \hspace{10pt} (s  \in   \C),
\end{align*}
respectively. Their inversion formulae, which hold whenever the integrals converge absolutely, are
\begin{align}\label{melT}
    \hspace{-20pt} \phi(x) =  \int_{\R}   \widehat{\phi}(y) e(xy) dy =:  \check{\widehat{\phi}}(x)  \hspace{10pt} \text{ and } \hspace{10pt}     \psi(x)  = \int_{(\sigma)}  \widetilde{\psi}(s) x^{-s}  \frac{ds}{2\pi i}.
   \end{align}

For $k\ge 1$, the $J$-Bessel functions can be defined by the series 
\begin{align}\label{powerserJ}
    J_{k-1}(z) \, := \,  \sum_{r=0}^{\infty} \, \frac{(-1)^r}{r! (r+k-1)!} \left(\frac{z}{2}\right)^{k-1+2r},
\end{align}
which converges pointwise absolutely and uniformly on compact subsets of $\C$. From  \cite[Sect. p. 192]{Wa95}, they admit the following absolutely convergent Mellin--Barnes representation:
\begin{lemma}\label{Melexpthm}
    Let $k\ge 2$ and $\gamma_{k}(s)$ be defined in (\ref{gl2gammfac}). Then 
    \begin{align}\label{bessgamm}
 	J_{k-1}(4\pi x)  =   \frac{1}{2\pi} \, \int_{(\sigma)}  \frac{\gamma_{k}(1-s)}{\gamma_{k}(s)}  x^{2(s-1)} \frac{ds}{2\pi i},  \hspace{20pt} x>0,  \hspace{10pt} 1< \sigma<(k+1)/2.
 \end{align}
\end{lemma}

 Let $F\in C_{c}^{\infty}(0, \infty)$ and $k\ge 2$.   Then the  \textit{Hankel transform} of $F$ is defined by
\begin{align}\label{Hatrasdef}
	(\mathcal{H}_{k}F)(a)  &:=   2\pi \int_{0}^{\infty} F(x) J_{k-1}(4\pi\sqrt{ax})  dx \hspace{20pt} (a \, > \,  0).
\end{align}
The rapid decay of $\mathcal{H}_{k}F$ can be deduced by integrating-by-parts many times in (\ref{Hatrasdef}) using
     \begin{align}\label{asympJbess}
        J_{k-1}\left( 2\pi x\right) =  W_{k}(x) e(x)   +  \overline{W_{k}}(x) e(-x) \hspace{15pt} (x>0),
    \end{align}
     where $W_{k}$ is a smooth function satisfying $x^{j}(\partial^{j}W_{k})(x) \ll_{j,k} 1/\sqrt{x}$ for any $j\ge 0$; see \cite[p. 206]{Wa95}.

 \begin{lemma}\label{lemderHank}
For $k> 2$ and $j\in \{0,1,2,3\}$ we have
\begin{align}\label{derHank}
  (\mathcal{H}_{k}F)^{(j)}(a) \ \ll_{k} \ a^{(k-2j-1)/2} \hspace{20pt} \text{ for } \hspace{10pt} 0 <a<1. 
    \end{align}
    
\end{lemma}   

This follows directly from the recurrence $2 J_{k}'(z) =J_{k-1}(z) - J_{k+1}(z)$, and the estimate
    \begin{align}\label{Jbessunibd}
        J_{k-1}(y) \ \ll_{k} \  y^{k-1} \hspace{10pt} \text{ for } \hspace{10pt} y \ > \ 0.
    \end{align}

From \cite[Sects.14.3--4]{Wa95} and \cite[Chp. 13.31]{Wa95}, we have the following well-known results.
\begin{lemma}[Hankel inversion formula]\label{hankelinv}
For any  $F\in C_{c}^{\infty}(0, \infty)$, we have
\begin{align}
	(\mathcal{H}_{k} \circ \mathcal{H}_{k}F)(b) \ = \  F(b) \hspace{20pt} (b \ > \  0). 
\end{align}
\end{lemma}

\begin{lemma}\label{keyspec}
Let  $k \ge 2$, $\re \alpha>0$, and  $\beta, \gamma>0$. Then
\begin{align}\label{keyspecinst}
    \int_{0}^{\infty} e^{-2\pi \alpha y} J_{k-1}(4\pi \beta\sqrt{y}) J_{k-1}(4\pi \gamma\sqrt{y}) dy  &=  \frac{i^{1-k}}{2\pi \alpha}\,  J_{k-1}\left( \frac{4\pi i\beta\gamma}{\alpha}\right) \exp\left(- 2\pi  \frac{\beta^2+\gamma^2}{\alpha}\right). 
\end{align}   
\end{lemma}

\begin{remark}
    While  Bessel functions with \emph{positive arguments} are more common in the analytic theory of automorphic forms (e.g., (\ref{vorformst}) or (\ref{Peter})), it will be technically convenient to invoke those with \emph{complex arguments} as intermediates in our case. Also,  the integral identity (\ref{keyspecinst}) is a variant of (\ref{Web}), where the latter only converges conditionally. 
\end{remark}

\begin{lemma}\cite[Lem. 8.1]{BKY}\label{BKYibp}
    Let $h\in C^{\infty}[\alpha, \beta]$ be a real-valued function and $w\in C_{c}^{\infty}[\alpha, \beta]$. Suppose there exist $W, V, H,G, R>0$ such that for any  $t\in [\alpha, \beta]$, we have $w^{(j)}(t) \, \ll_{j} \,  W/V^{j}$ for any $j\ge 0$, $h^{(j)}(t) \ll_{j} H/G^{j}$ for any $j\ge 2$, and $|h'(t)| \, \ge \, R$.  Then for any $A\ge  0$, we have
\begin{align}
    \int_{\R} \, w(t) e\left(h(t)\right) \ dt \ \ll_{A}  \  \left(\beta -\alpha\right) W \left(  \frac{1}{RV} \ + \  \frac{1}{RG}  \ + \  \frac{H}{(RG)^2}    \right)^A.  
\end{align}
\end{lemma}

From \cite[Prop. 14.5 and Lem. 14.10]{IK} and \cite[Chp. 4.3]{IK}, we have the following:
\begin{lemma}[Petersson trace formula]\label{Peterform}
Let $\mathcal{B}_{k}(1)$ be an orthogonal basis of holomorphic cuspidal Hecke eigenforms of level $1$ and weight $k$.  For any $\ell, n\ge 1$, we have  \footnote{ In this article, we use $\delta(\cdots)$ to denote the indicator function with respect to the condition $(\cdots)$. }
    \begin{align}
	\sideset{}{^h}{\sum}_{f \in \mathcal{B}_{k}(1)} \ \lambda_{f}(\ell) \lambda_{f}(n) \ =  \ \delta(n=\ell) \ + \  2\pi i^{-k} \, \sum_{c=1}^{\infty} \ \frac{S(n, \ell;c)}{c}\,  J_{k-1}\left(\frac{4\pi\sqrt{n\ell}}{c}\right).  \label{Peter}
\end{align}
\end{lemma}

\begin{lemma}[Poisson summation]\label{lemPois}
Let $c, X>0$ and $c\in \Z$. For  $V\in C_{c}^{\infty}(\R)$, and $K: \Z\rightarrow \mathbb{C}$ be $c$-periodic,
	\begin{align}
			\sum_{n\in \Z}  V(n/X) K(n)  =   \frac{X}{c}  \sum_{m\in \Z}  \,  \bigg( \, \sum_{\gamma (c)} \ K(\gamma) \, e_c\left(m\gamma\right)\bigg)    \int_{0}^{\infty}  V(y) e\left(-\frac{mXy}{c}\right) dy. 
	\end{align}
\end{lemma}


\section{Setting the stage: Petersson--Poisson--Reciprocity}\label{PPR}

Sects. \ref{PPR}--\ref{Backwd} are devoted to proving Thm. \ref{ourmainthmtwist}. Let $g\in C_{c}^{\infty}(0,\infty)$ be a given test function. Define
\begin{align}\label{truncIsum}
I_{k}(\ell; \chi ) \ :=  \ \sum_{n=1}^{\infty}  \ g(n) \chi(n) \  \sideset{}{^h}{\sum}_{f \in \mathcal{B}_{k}(1)}  \lambda_{f}(\ell) \lambda_{f}(n).
\end{align} 
 Our analysis of this article applies to all even integers $k\ge 6$,  though it is well-known by the Riemann--Roch theorem that there is no non-zero holomorphic cusp form of level $1$ and weight $k<12$.


\subsection{Step 1: Petersson trace formula}\label{Petersect}

Applying \eqref{Peter} to (\ref{truncIsum}), followed by opening up the Kloosterman sums by its definition, we obtain
\begin{align}\label{finiPet}
		I_{k}(\ell; \chi )  =   g(\ell)\chi(\ell) \ + \ 2\pi i^{-k} \sum_{c=1}^{\infty} \ c^{-1} \ \sideset{}{^*}{\sum}_{x \, (c)}  \, e_c\left(\ell \overline{x}\right)\sum_{n} \ g(n)\chi(n) \, J_{k-1}\left(\frac{4\pi\sqrt{n\ell}}{c}\right) \, e_{c}(nx). 
\end{align}

\subsection{Step 2: Poisson summation}\label{sect:step2}

We apply Lem. \ref{lemPois} to the $n$-sum of (\ref{finiPet}), which gives
\begin{align*}
		I_{k}(\ell; \chi )  = & \ g(\ell) \chi(\ell)\, + \, 2\pi i^{-k} \sum_{c=1}^{\infty} \, \frac{1}{c} \, \sideset{}{^*}{\sum}_{x \, (c)}   e_c\left(\ell \overline{x}\right) \cdot \frac{1}{cq}\,  \sum_{m} \ \bigg(\, \sum_{\gamma\, (cq)}  \chi(\gamma)e_c\left(\gamma x\right)  e_{cq}\left(m\gamma\right)\bigg) \int_{0}^{\infty} g(y) J_{k-1}\bigg(\frac{4\pi\sqrt{y\ell}}{c}\bigg) e\bigg(-\frac{my}{cq}\bigg) \, dy.
\end{align*}
The $\gamma$-sum can be decomposed via $\gamma= \alpha+\beta q$ with $\alpha \, (\bmod\, q)$ and $\beta \, (\bmod\, c)$, i.e., \begin{align}\label{FirstPoissonCharComp}
    \sum_{\gamma\, (cq)} \ \chi(\gamma)e_c\left(\gamma x\right)  e_{cq}\left(m\gamma\right) \ =& \ \sum_{\alpha\,(q)}\sum_{\beta\,(c)}\chi(\alpha)e_{cq}\left((\alpha+\beta q)(qx+m)\right)\nonumber\\
     \ =& \ c\, \delta(xq\equiv -m\,(c))\sum_{\alpha\,(q)}\chi(\alpha)e_q\left(\alpha(qx+m)/c\right)\nonumber\\
    \ =& \ c\sqrt{q}\epsilon_\chi\delta(xq\equiv -m\,(c))\overline{\chi}\left(\frac{qx+m}{c}\right),
\end{align}
where  the last line follows from the primitivity of  $\chi \, (\bmod\, q)$; see \cite[eq. (3.12)]{IK}. Hence, we get \begin{align}\label{IafterPoisson}
		I_{k}(\ell; \chi )   =    g(\ell) \chi(\ell) 
  \, + \,   2\pi i^{-k} \frac{\epsilon_\chi}{\sqrt{q}}\sum_{c=1}^{\infty} \,  c^{-1} \ \sideset{}{^*}{\sum}_{x \, (c)}  \, e_c\left(\ell \overline{x}\right) \, & \sum_{m\equiv -xq\, (c)} \, \overline{\chi}\left(\frac{qx+m}{c}\right)\int_{0}^{\infty} \,  g(y)  J_{k-1}\bigg(\frac{4\pi\sqrt{y\ell}}{c}\bigg) e\bigg(-\frac{my}{cq}\bigg)  dy.
\end{align}


We split the $c$-sum of (\ref{IafterPoisson}) by $c=c_0c'$, where $c_0:=(c,q^\infty)$ and $(c',q)=1$. In particular, we have
\begin{align}\label{FirstPoissonCharComp2}
    \sideset{}{^*}{\sum}_{\substack{x \, (c)\\m\equiv -xq\, (c)}}  \, e_c\left(\ell \overline{x}\right)\overline{\chi}\left(\frac{qx+m}{c}\right) \ = \ \chi(c')e_{c'}\left(-\ell q\overline{c_0m}\right)\sumast_{\substack{\alpha\, (c_0)\\ m\equiv-\alpha q\, (c_0)}}e_{c_0}\left(\ell\overline{\alpha c'}\right)\overline{\chi}\bigg(\frac{\alpha q+m}{c_0}\bigg).
\end{align}

Notice that the dual zeroth frequency (i.e., the term $m=0$) of (\ref{IafterPoisson})  is non-vanishing only when $c_0|q$ and $c'=1$.
In this case, the $\alpha$-sum is given by \begin{align*}
    \sumast_{\alpha\, (c_0)}e_{c_0}\left(\ell\overline{\alpha }\right)\overline{\chi}\left(\alpha \frac{q}{c_0}\right) \ = \ \delta(c_0=q)\sum_{\alpha\,(q)}\overline{\chi}(\alpha)e_q\left(\ell\overline{\alpha}\right) \ = \ \sqrt{q}\epsilon_{\chi}\overline{\chi}(\ell).
\end{align*}
We extract the zeroth frequency from the rest of the frequencies, resulting in the expression:
\begin{align}\label{dualPoisspl}
		I_{k}(\ell; \chi )  \, = \,  g(\ell)\chi(\ell) \ &+ \ 2\pi i^{-k}\frac{\epsilon_\chi^2}{q}\overline{\chi}(\ell)	\int_{0}^{\infty} \,  g(y)J_{k-1}\bigg(\frac{4\pi\sqrt{y\ell}}{q}\bigg)  \, dy \, + \,  S_k(\ell; \chi),
\end{align}
where \begin{align}\label{SkDef}
    S_k(\ell; \chi)\ := \ &\ 2\pi i^{-k} \frac{\epsilon_\chi}{\sqrt{q}}\sum_{c_0|q^\infty}\sum_{\substack{c'\geq1\\ (c',q)=1}} \sum_{\substack{m\neq0 \\ (m,c')=1}} \  \frac{\chi(c')}{c_0c'} \, e_{c'}\left(-\ell q\overline{c_0m}\right)\sumast_{\substack{\alpha\, (c_0)\\ m\equiv-\alpha q\, (c_0)}}e_{c_0}\left(\ell\overline{\alpha c'}\right)\overline{\chi}\left(\frac{\alpha q+m}{c_0}\right)
    \nonumber\\
    & \hspace{7.5cm}\times \int_{0}^{\infty} \,  g(y) \, J_{k-1}\bigg(\frac{4\pi\sqrt{y\ell}}{c_0c'}\bigg) e\bigg(-\frac{my}{c_0c'q}\bigg) \, dy.
\end{align}


\subsection{Step 3: Additive reciprocity.}\label{Recip} 

Applying  (\ref{reciprel}) to the factor $ e_{c'}\left(-\ell q\overline{c_0m}\right)$ in (\ref{SkDef}), it follows that
\begin{align}\label{afterrep}
    S_k(\ell; \chi)  \ =\ & \  2\pi i^{-k} \frac{\epsilon_\chi}{\sqrt{q}}\sum_{c_0|q^\infty} \sum_{\substack{c'\geq1\\ (c',q)=1}} \sum_{\substack{m\neq0 \\ (m,c')=1}}  \  \frac{\chi(c')}{c_0c'}  \ e\left(\frac{\ell q\overline{c'}}{c_0m}-\frac{\ell q}{c_0c'm}\right)\sumast_{\substack{\alpha\, (c_0)\\ m\equiv-\alpha q\, (c_0)}}e_{c_0}\left(\ell\overline{\alpha c'}\right)\overline{\chi}\left(\frac{\alpha q+m}{c_0}\right)
    \nonumber\\
    &\hspace{7cm}\times \int_{0}^{\infty} \,  g(y) J_{k-1}\bigg(\frac{4\pi\sqrt{y\ell}}{c_0c'}\bigg) e\bigg(-\frac{my}{c_0c'q}\bigg) \, dy.
\end{align}


\section{Step 4: Preparation on the \texorpdfstring{$c_0,c'$}{c0,c'}-sums}\label{applyhankcanc}

To prepare for Poisson summation in the $c_0,c'$-sums and for swapping the roles of $m$ and $c_0c'$ at a later stage, two crucial components are needed:
\begin{itemize}
    \item (Analytic) Transform the $c'$-sum with $c'\geq1$ in (\ref{afterrep}) to a sum  over all integers, and remove the singularity at $c'=0$ by suitable analytic manipulations;

    \item (Arithmetic) Transform the character sum in (\ref{afterrep}) into a suitable form that facilitates the eventual combination of $c_0$- and $c'$-sum. 
\end{itemize}
\noindent The second point requires a fairly intricate analysis of character sums, see Sects. \ref{arithpre}--\ref{reciprecom}.


\subsection{Step 4.1: Analytic preparation}\label{Hankana}

  The $c'$-sum of (\ref{afterrep}) can be rewritten to sum over all non-zero integers. Indeed, observe that the contribution from $m<0$ in \eqref{afterrep} is given by \begin{align}
     &\hspace{20pt} 2\pi i^{-k} \frac{\epsilon_\chi}{\sqrt{q}}\sum_{c_0|q^\infty}\sum_{\substack{c'\geq1\\(c',q)=1}} \, \sum_{\substack{m\geq1 \\ (m,c')=1}} \  \frac{\chi(c')}{c_0c'} \, e\left(-\frac{\ell q\overline{c'}}{c_0m}+\frac{\ell q}{c_0c'm}\right)\sumast_{\substack{\alpha\, (c_0)\\ m\equiv-\alpha q\, (c_0)}}e_{c_0}\left(-\ell\overline{\alpha c'}\right)\overline{\chi}\left(-\frac{\alpha q+m}{c_0}\right)
    \nonumber\\
    &\hspace{7cm}\times \int_{0}^{\infty} \,  g(y) \, J_{k-1}\left(\frac{4\pi\sqrt{y\ell}}{c_0c'}\right) e\left(\frac{my}{c_0c'q}\right) \ dy, \label{negapartSk}
\end{align}
upon making the changes of variables $m\mapsto -m$ and $\alpha\mapsto -\alpha$. The change of variables $c'\mapsto -c'$ and the fact that $J_{k-1}(-x) = -J_{k-1}(x)$ (follows from (\ref{powerserJ}) and $k\in 2\mathbb{N}$) allow us to rewrite (\ref{negapartSk}) as: 
\begin{align*}
   &2\pi i^{-k} \frac{\epsilon_\chi}{\sqrt{q}}\sum_{c_0|q^\infty} \sum_{\substack{c'\leq-1\\(c',q)=1}} \sum_{\substack{m\geq1 \\(m,c')=1}} \  \frac{\chi(c')}{c_0c'} \, e\left(\frac{\ell q\overline{c'}}{c_0m}-\frac{\ell q}{c_0c'm}\right)\sumast_{\substack{\alpha\, (c_0)\\ m\equiv-\alpha q\, (c_0)}}e_{c_0}\left(\ell\overline{\alpha c'}\right)\overline{\chi}\left(\frac{\alpha q+m}{c_0}\right)
    \nonumber\\
    &\hspace{6.2cm}\times \int_{0}^{\infty} \,  g(y) \, J_{k-1}\left(\frac{4\pi\sqrt{y\ell}}{c_0c'}\right) e\left(-\frac{my}{c_0c'q}\right) \ dy.
\end{align*}
Hence, the expression (\ref{afterrep}) becomes: \begin{align}\label{signpiece}
    S_k(\ell; \chi)  \ = \ &\  2\pi i^{-k} \frac{\epsilon_\chi}{\sqrt{q}}\sum_{c_0|q^\infty}\sum_{\substack{c'\neq0\\ (c',q)=1}} \sum_{\substack{m\geq1 \\ (m,c')=1}} \  \frac{\chi(c')}{c_0c'}  \, e\left(\frac{\ell q\overline{c'}}{c_0m}-\frac{\ell q}{c_0c'm}\right)\sumast_{\substack{\alpha\, (c_0)\\ m\equiv-\alpha q\, (c_0)}}e_{c_0}\left(\ell\overline{\alpha c'}\right)\overline{\chi}\left(\frac{\alpha q+m}{c_0}\right)
    \nonumber\\
    &\hspace{5.5cm}\times \int_{0}^{\infty} \,  g(y) \, J_{k-1}\left(\frac{4\pi\sqrt{y\ell}}{c_0c'}\right) e\left(-\frac{my}{c_0c'q}\right) \ dy.
\end{align}

  The last integral converges absolutely since $g\in C_{c}^{\infty}(0, \infty)$.  By the dominated convergence theorem, 
\begin{align*}
  \int_{0}^{\infty} g(y) J_{k-1}\left(\frac{4\pi\sqrt{y\ell}}{c_0c'}\right)   e\left(-\frac{my}{c_0c'q}\right) dy =  \lim_{\epsilon\to 0+}   \int_{0}^{\infty}  g(y)  J_{k-1}\left(\frac{4\pi\sqrt{y\ell}}{c_0c'}\right) \exp\left(-2\pi\left(\epsilon+ \frac{im}{c_{0}c'q}\right)y\right)  dy. 
\end{align*}
The Hankel inversion formula (Lem. \ref{hankelinv}) with the change of variables $x\to x/q^2$ give
\begin{align}\label{hankinter}
  \hspace{5pt} \int_{0}^{\infty}  g(y) & J_{k-1}\left(\frac{4\pi\sqrt{y\ell}}{c_0c'}\right)  e\left(-\frac{my}{c_0c'q}\right) dy \nonumber\\
   \ &\hspace{-30pt}=   \frac{2\pi}{q^2}  \lim_{\epsilon\to 0+}   \int_{0}^{\infty} (\mathcal{H}_{k}g)\left(x/q^2\right)  \int_{0}^{\infty}  J_{k-1}\left(\frac{4\pi\sqrt{xy}}{q}\right)   J_{k-1}\left(\frac{4\pi\sqrt{y\ell}}{c_0c'}\right) \exp\left(-2\pi\left(\epsilon+ \frac{im}{c_{0}c'q}\right)y\right) 
  dy  dx,
\end{align}
where the interchange of the order of integration is permitted by absolute convergence and the decay of the Hankel transform $\mathcal{H}_{k}g$ of $g$.  Now, we are in a position to apply Lem.\ref{keyspec}. In other words,
\begin{align}
   \hspace{10pt}   \int_{0}^{\infty} g(y)  & J_{k-1}\left(\frac{4\pi\sqrt{y\ell}}{c_0c'}\right)   e\left(-\frac{my}{c_0c'q}\right)  dy \nonumber\\ 
      &\hspace{-30pt}=  \frac{2\pi}{q^2}  \lim_{\epsilon\to 0+}   \int_{0}^{\infty}  (\mathcal{H}_{k}g)\left(x/q^2\right)  \frac{i^{1-k}}{2\pi \left(\epsilon+ \frac{im}{c_{0}c'q}\right)} J_{k-1}\left( \frac{4\pi i\sqrt{\ell x}}{c_{0}c'q \epsilon +im }\right) \exp\left(- 2\pi \, \frac{x/q^2+\ell/c^2}{\epsilon+ \frac{im}{c_{0}c'q}}\right)  dx.
\end{align}
By the decay of  $\mathcal{H}_{k}g$ and dominated convergence again, it follows from continuity that
\begin{align}
     \hspace{5pt} \int_{0}^{\infty} g(y) J_{k-1}\left(\frac{4\pi\sqrt{y\ell}}{c_0c'}\right)  e\left(-\frac{my}{c_0c'q}\right) dy 
     \,  = \,  \frac{i^{-k}}{q}  \frac{c_{0}c'}{m}\, e\left(\frac{\ell q}{c_{0}c'm}\right) \int_{0}^{\infty}  (\mathcal{H}_{k}g)\left(x/q^2\right)   J_{k-1}\left(\frac{4\pi\sqrt{\ell x}}{m}\right) e\left(\frac{c_{0}c'x}{qm}\right) dx.
\end{align}
Substitute the last expression back into (\ref{signpiece}) and observe the cancellations in (1). a pair of exponential phases, and (2). a pair of factors $c_{0}c'$.  Thus, we obtain
    \begin{align}\label{SkAfterAnalyticPrep}
         S_k(\ell; \chi)  \, = \,   \frac{2\pi}{q} \frac{\epsilon_\chi}{\sqrt{q}}  \sum_{c_0|q^\infty}  \sum_{\substack{c'\neq0\\ (c',q)=1}} \   \sum_{\substack{m\geq1 \\ (m,c')=1}} \,  \frac{1}{m}   \chi(c') & \, e_{c_0m}\left(\ell q\overline{c'}\right)\, \sumast_{\substack{\alpha\, (c_0)\\ m\equiv-\alpha q\, (c_0)}}e_{c_0}\left(\ell\overline{\alpha c'}\right)\overline{\chi}\left(\frac{\alpha q+m}{c_0}\right) \nonumber\\
         &\times  \,    \int_{0}^{\infty} \ (\mathcal{H}_{k}g)\left(x/q^2\right)   J_{k-1}\left(\frac{4\pi\sqrt{\ell x}}{m}\right) e\left(\frac{c_{0}c'x}{qm}\right)  \, dx.
    \end{align}


We must show that the order of summation can be interchanged for Sect. \ref{reversPois}. This requires $k\ge 6$ and follows from trivially bounding the sums (the $\alpha$-sum by $c_{0}$) and integration by parts thrice in (\ref{SkAfterAnalyticPrep}), cf. \cite{KLev}. Indeed,  by Lem. \ref{lemderHank} and (\ref{asympJbess}), observe that
\begin{align}
   \left(\frac{c_{0}c'}{m}\right)^{3} \int_{0}^{\infty} \, (\mathcal{H}_{k}g)\left(x/q^2\right)   J_{k-1}\left(\frac{4\pi\sqrt{\ell x}}{m}\right) e\left(\frac{c_{0}c'x}{qm}\right)  \, dx \, \ll_{\ell, q, k} \, \frac{1}{m^{k-1}}.
\end{align}


\subsection{Step 4.2:\,   Arithmetic preparation}\label{arithpre}

We now analyze the arithmetic component of (\ref{SkAfterAnalyticPrep}), i.e., \begin{align}\label{char1}
    \chi(c')  e_{c_0m}\left(\ell q\overline{c'}\right)\sumast_{\substack{\alpha\, (c_0)\\ m\equiv-\alpha q\, (c_0)}}e_{c_0}\left(\ell\overline{\alpha c'}\right)\overline{\chi}\left(\frac{\alpha q+m}{c_0}\right).
\end{align}
For technical convenience, we consider a more general character sum define as follows.  Let $a,b,h,r,u,v$ be integers such that $ab|r^\infty$ and $(uv,r)=1$. Let $\psi$ be a primitive character mod $r$. Define \begin{align}\label{keycharsum}
    \mathcal{C}_\psi^h(a,u,b,v):=\psi(u) \sumast_{\substack{\alpha\, (a)\\ bv\equiv-\alpha r\, (a)}}e_a\left(h\overline{\alpha u}\right)\overline{\psi}\left(\frac{\alpha r+bv}{a}\right).
\end{align}
We decouple the $m$-sum of (\ref{SkAfterAnalyticPrep}) by $m=m_0m'$ (as before), where $m_0:=(m,q^\infty)$ and $(m',q)=1$. Then \eqref{char1} is equal to $e_{c_0m_0m'}(\ell q\overline{c'})\mathcal{C}_\chi^\ell(c_0,c',m_0,m'),$ and
\begin{align}\label{roleofkeychar}
     S_k(\ell; \chi) \ = \   & \frac{2\pi}{q} \frac{\epsilon_\chi}{\sqrt{q}} \, \sum_{c_0|q^\infty} \, \sum_{m_{0} \mid q^{\infty}} \, \sum_{\substack{m'\ge 1\\ (m',q)=1}} \, \frac{1}{m_{0}m'}\sum_{\substack{c'\neq0\\ (c',m'q)=1}} \  e_{c_0m_0m'}\left(\ell q\overline{c'}\right)\mathcal{C}_\chi^\ell(c_0,c',m_0,m')
    \nonumber\\
    &\hspace{4cm} \times  \,    \int_{0}^{\infty} \ (\mathcal{H}_{k}g)\left(x/q^2\right)   J_{k-1}\left(\frac{4\pi\sqrt{\ell x}}{m_{0}m'}\right) e\left(\frac{c_{0}c'x}{qm_{0}m'}\right)  \ dx. 
\end{align}

We have the following twisted multiplicativity for the character sum (\ref{keycharsum}). 

\begin{lemma}\label{lemma:charmult}
    Let $a_1,a_2,b_1,b_2,h,r_1,r_2,u,v$ be integers such that $a_1b_1|r_1^\infty$, $a_2b_2|r_2^\infty$, and $(r_1,r_2)=(uv,r_1r_2)=1$. Let $\psi_1$ and $\psi_2$ be primitive characters mod $r_1$ and $r_2$ respectively. Then \begin{align*}
        \mathcal{C}_{\psi_1\psi_2}^h(a_1a_2,u,b_1b_2,v)=&\ \overline{\psi_1(b_2^2)\psi_2(b_1^2)}\mathcal{C}_{\psi_1}^{hr_2}(a_1,a_2b_2u,b_1,v)\mathcal{C}_{\psi_2}^{hr_1}(a_2,a_1b_1u,b_2,v)\\
        =&\ \overline{\psi_1(a_2^2)\psi_2(a_1^2)}\mathcal{C}_{\psi_1}^{hr_2}(a_1,u,b_1,a_2b_2v)\mathcal{C}_{\psi_2}^{hr_1}(a_2,u,b_2,a_1b_1v).
    \end{align*}
\end{lemma}

\begin{proof}
    By the Chinese remainder theorem, we have \begin{align*}
        \mathcal{C}_{\psi_1\psi_2}^h(a_1a_2,u,b_1b_2,v)=&\psi_1(a_2u)\sumast_{\substack{\alpha_1\Mod{a_1}\\ b_1b_2v\equiv -\alpha_1 r_1r_2\,(a_1)}}e_{a_1}(h\overline{\alpha_1 a_2u})\overline{\psi_1}\left(\frac{\alpha_1r_1r_2+b_1b_2v}{a_1}\right)\\
        &\times \psi_2(a_1u)\sumast_{\substack{\alpha_2\Mod{a_2}\\ b_1b_2v\equiv -\alpha_2 r_1r_2\,(a_2)}}e_{a_2}(h\overline{\alpha_2 a_1u})\overline{\psi_2}\left(\frac{\alpha_2r_1r_2+b_1b_2v}{a_2}\right).
    \end{align*}
    Applying the change of variables $\alpha_1\mapsto \alpha_1b_2\overline{r_2}$ and $\alpha_2\mapsto \alpha_2b_1\overline{r_1}$ yields the first equality. To get the second equality, apply the change of variables $\alpha_1\mapsto \alpha_1\overline{a_2r_2}$ and $\alpha_2\mapsto \alpha_2\overline{a_1r_1}$ instead.
\end{proof}

\begin{remark}
    This lemma explains why the factor $e_{abv}\left(\ell q\overline{u}\right)$ is purposefully left out when defining $\mathcal{C}_\psi^\ell(\cdots)$. Otherwise, no multiplicative relation would hold as the factor is not multiplicative in  $a$, $b$.

\end{remark}


\section{Local analysis of character sums: proof of Theorem \ref{Thm:CharSum} }\label{sect:ArithPrep}

\begin{namedthm}[Restatement of Theorem \ref{Thm:CharSum}]\label{lemma:CharSum}
    Let $a, b, h, r, u, v$ be integers such that $ab|r^\infty$ and $(uv,r)=1$. Let $\psi$ be a character mod $r$. Then we have \begin{align*}
        \mathcal{C}_\psi^h(a,u,b,v) \ = \ e_{ab}\left(-hr\overline{uv}\right)\overline{\mathcal{C}_\psi^h(b,v,a,u)}.
    \end{align*}
\end{namedthm}

\begin{proof}
    We first consider the case when $r=p^k$, where $p$ is a prime and $k\geq1$. Since $ab|r^\infty=p^\infty$, we may write $a=p^s$ and $b=p^t$ for some $s, t\geq0$. In this, we have 
    \begin{align*}
    \mathcal{C}_\psi^h(p^s,u,p^t,v) \ = \ \psi(u)\sumast_{\substack{\alpha\, (p^s)\\ p^tv\equiv-\alpha p^k\, (p^s)}}e_{p^s}\left(h\overline{\alpha u}\right)\overline{\psi}\left(\frac{\alpha p^k+p^tv}{p^s}\right),
    \end{align*}
     \begin{align*}
        e_{p^{s+t}}\left(-h p^k\overline{uv}\right)\overline{\mathcal{C}_\psi^h(p^t,v,p^s,u)} \ =& \ \overline{\psi}(v) e_{p^{s+t}}\left(-h p^k\overline{uv}\right)\sumast_{\substack{\alpha\, (p^t)\\ p^su\equiv-\alpha p^k\, (p^t)}}e_{p^t}\left(-h\overline{\alpha v}\right)\psi\left(\frac{\alpha p^k+p^su}{p^t}\right).
    \end{align*}
    
    \textbf{Claim:}  Unless one of the following holds: 
\begin{enumerate*}
    \item  $s=t\leq k$,
    
    \item  $t=k<s$, or

    \item   $s=k<t$, 
\end{enumerate*} we have
    \begin{align*}
        \mathcal{C}_\psi^h(p^s,u,p^t,v)  =  0  = \overline{\mathcal{C}_\psi^h(p^t,v,p^s,u)}.
    \end{align*}

    Indeed, the congruence condition for $\mathcal{C}_\psi^h(p^s,u,p^t,v)$ implies that it is equal to $0$ unless one of the following holds: $s\leq t\leq k$, (2) or (3). If $s<t\leq k$, then
    \begin{align*}
        \overline{\psi}\left(\frac{\alpha p^k+p^tv}{p^s}\right) \ = \ \overline{\psi}\left(\alpha p^{k-s}+p^{t-s}v\right) \ = \ 0, 
    \end{align*}
    and hence $\mathcal{C}_\psi^h(p^s,u,p^t,v)=0$ unless (1) or (2) or (3) holds. Similarly, the congruence condition implies that $\overline{\mathcal{C}_\psi^h(p^t,v,p^s,u)}=0$ unless $t\leq s\leq k$ or (2) or (3), and the presence of $\psi\left(\frac{\alpha p^k+p^su}{p^t}\right)$ implies that it is also $0$ if $t<s\leq k$. This proves our claim. 

     We apply a change of variable $\alpha\mapsto \alpha v$ to obtain 
    \begin{align*}
        \mathcal{C}_\psi^h(p^s,u,p^t,v) \ = \ \psi(u\overline{v}) \sumast_{\substack{\alpha\, (p^s)\\ p^t\equiv-\alpha p^k\, (p^s)}}e_{p^s}\left(h \overline{\alpha uv}\right)\overline{\psi}\left(\frac{\alpha p^k+p^t}{p^s}\right).
    \end{align*}

    \noindent \underline{Case 1:} $s=t\leq k$. In this case, we have \begin{align*}
        \mathcal{C}_\psi^h(p^s,u,p^s,v)=&\ \psi(u\overline{v})\sumast_{\substack{\alpha\, (p^s)}}e_{p^s}\left(h \overline{\alpha uv}\right)\overline{\psi}\left(\alpha p^{k-s}+1\right)\\
        =&\ \psi(u\overline{v})e_{p^{2s}}\left(-hp^k \overline{uv}\right)\sumast_{\substack{\alpha\, (p^s)}}e_{p^s}\left(h \overline{uv}(\overline{\alpha}+p^{k-s})\right)\overline{\psi}\left(\alpha p^{k-s}+1\right).
    \end{align*}
    Notice that since $p\nmid \alpha$, we have $p^{k-s}||(\overline{\alpha p^{k-s}+1}-1)$. Hence, the change of variables $\beta = (\overline{\alpha p^{k-s}+1}-1)/p^{k-s}$ is admissible.
    When $k=s$, the quantity $\overline{\alpha+1}$ is well-defined due to the presence of $\psi$. With the above change of variables $(\bmod\, p^s)$, it follows that
    \begin{align*}
            \overline{\alpha p^{k-s}+1} \ \equiv \ \beta p^{k-s}+1 \, \Mod{p^s},
    \end{align*}
    \begin{align*}
        \overline{\alpha}  \equiv\overline{(\overline{\beta p^{k-s}+1}-1)/p^{k-s}} 
         \equiv   \overline{(1-\beta p^{k-s}-1)\overline{\beta p^{k-s}+1}/p^{k-s}} 
      \equiv  -\overline{\beta}(\beta p^{k-s}+1)  \equiv  -p^{k-s}-\overline{\beta}\, \Mod{p^s}.
    \end{align*}
    Hence, \begin{align*}
        \mathcal{C}_\psi^h(p^s,u,p^s,v)= \psi(u\overline{v})e_{p^{2s}}\left(-h p^k\overline{uv}\right)\sumast_{\substack{\beta\, (p^s)}}e_{p^s}\left(-h \overline{\beta uv}\right)\psi\left(\beta p^{k-s}+1\right).
    \end{align*}
    With another change of variables $\beta=\gamma\overline{u}$, we get \begin{align*}
        \mathcal{C}_\psi^h(p^s,u,p^s,v) \ = \ e_{p^{2s}}\left(-h p^k\overline{uv}\right)\overline{\mathcal{C}_\psi^h(p^s,v,p^s,u)}.
    \end{align*}

    \noindent \underline{Case 2:} $t=k<s$. In this case, we have \begin{align*}
        \mathcal{C}_\psi^h(p^s,u,p^t,v) \ = \  \psi(u\overline{v}) \sumast_{\substack{\alpha\, (p^s)\\ \alpha\equiv-1 \, (p^{s-t})}}e_{p^s}\left(h \overline{\alpha uv}\right)\overline{\psi}\left(\frac{\alpha +1}{p^{s-t}}\right).
    \end{align*}
    Apply the change of variables $p^{s-t}\beta=\alpha+1$ and observe the fact that $s>t$, we have \begin{align*}
        \mathcal{C}_\psi^h(p^s,u,p^t,v) \ = \  \psi(u\overline{v}) \sumast_{\substack{\beta\, (p^t)}}e_{p^s}\left(h \overline{uv(p^{s-t}\beta-1)}\right)\overline{\psi}\left(\beta\right).
    \end{align*}
    Another change of variables $\beta=\overline{\gamma+p^{s-t}}$, which is admissible as $s>t$, yields 
    \begin{align*}
        \mathcal{C}_\psi^h(p^s,u,p^t,v) \ = \  \psi(\overline{v}) \sumast_{\substack{\gamma\, (p^t)}}e_{p^s}\left(h \overline{uv(p^{s-t}\overline{\gamma+p^{s-t}}-1)}\right)\psi\left(\gamma u+p^{s-t}u\right).
    \end{align*}
    Using \begin{align*}
        p^{s-t} \, \overline{\gamma+p^{s-t}}-1\equiv   p^{s-t}\overline{\gamma}\overline{1+p^{s-t}\overline{\gamma}} \ - \ 1 
        \equiv  (1+p^{s-t}\overline{\gamma}-1)\overline{1+p^{s-t}\overline{\gamma}} \ - \ 1
        \equiv  -\overline{{1+p^{s-t}\overline{\gamma}}} \, \Mod{p^s},
    \end{align*}
    we arrive at \begin{align*}
        \mathcal{C}_\psi^h(p^s,u,p^t,v) \ = \  \psi(\overline{v}) \sumast_{\substack{\gamma\, (p^t)}}e_{p^s}\left(-h \overline{uv}(1+p^{s-t}\overline{\gamma})\right)\psi\left(\gamma u+p^{s-t}u\right).
    \end{align*}
    A final change of variables $\gamma\mapsto \gamma \overline{u}$ implies
    \begin{align*}
        \mathcal{C}_\psi^h(p^s,u,p^t,v) \ = \ e_{p^s}\left(-h \overline{uv}\right)\overline{\mathcal{C}_\psi^h(p^t,v,p^s,u)}.
    \end{align*}

    \noindent \underline{Case 3:} $s=k<t$. In this final case, we make use of Case 2 to deduce the answer.  For $s=k<t$, we have \begin{align*}
        \mathcal{C}_\psi^h(p^t,v,p^s,u)=e_{p^t}\left(-h \overline{uv}\right)\overline{\mathcal{C}_\psi^h(p^s,u,p^t,v)},
    \end{align*}
    and this yields \begin{align*}
        \mathcal{C}_\psi^h(p^s,u,p^t,v)=e_{p^t}\left(-h \overline{uv}\right)\overline{\mathcal{C}_\psi^h(p^t,v,p^s,u)}.
    \end{align*}

    Combining all three cases together, we have proved that for $r=p^k$ for some prime $p$ and $k\geq1$, 
    \begin{align}\label{primepoweq}
        \mathcal{C}_\psi^h(a,u,b,v) \ = \ e_{ab}\left(-h r\overline{uv}\right)\overline{\mathcal{C}_\psi^h(b,v,a,u)}
    \end{align}
    for any $a,b,u,v$ with $ab|r^\infty$, $(uv,r)=1$ and any  character mod $r$. Finally, applying the first and second equality of Lem. \ref{lemma:charmult} to the left-hand and right-hand side of (\ref{primepoweq}) respectively, with the observation:
     \begin{align}\label{anotrecip}
        e\left(-\frac{h r\overline{uv}}{a_1a_2b_1b_2}\right) \ = \ e\left(-\frac{h r\overline{a_2b_2uv}}{a_1b_1}-\frac{h r\overline{a_1b_1uv}}{a_2b_2}\right), 
    \end{align}
    we see that both sides satisfy the same multiplicative relations. This yields the result for general $r$.
\end{proof}


\section{Backward maneuver: Finishing the Proof of Theorem \ref{ourmainthmtwist}}\label{Backwd}

\subsection{Step 5: Reciprocity and recombining sums}\label{reciprecom}

From Thm. \ref{Thm:CharSum} and additive reciprocity, observe that
\begin{align}
 e_{c_0m_0m'}\left(\ell q\overline{c'}\right)\mathcal{C}_\chi^\ell(c_0,c',m_0,m') 
    =&\ e\left(\frac{\ell q\overline{c'}}{c_0m_0m'}-\frac{\ell q\overline{c'm'}}{c_0m_0}\right)\overline{\mathcal{C}_\chi^\ell(m_0,m',c_0,c')}\nonumber\\
    \ =& \ e_{m'}\left(\ell q\overline{c_0c'm_0}\right)\overline{\mathcal{C}_\chi^\ell(m_0,m',c_0,c')}. \label{COstru}
\end{align}
  Inserting this back into (\ref{roleofkeychar}) and open up $\overline{\mathcal{C}_\chi^\ell}(\cdots)$ by its definition, we obtain \begin{align*}
    S_k(\ell; \chi)  \ = \ & \  \frac{2\pi}{q} \frac{\epsilon_\chi}{\sqrt{q}} \, \sum_{c_0, m_0|q^\infty}\sum_{\substack{m'\geq1\\(m',q)=1}}\sum_{\substack{c'\neq0\\ (c',m'q)=1}} \  \frac{\overline{\chi}(m')}{m_0m'}  e_{m'}\left(\ell q\overline{c_0c'm_0}\right)\sumast_{\substack{\alpha\, (m_0)\\ c_0c'\equiv-\alpha q\, (m_0)}}e_{m_0}\left(-\ell\overline{\alpha m'}\right)\chi\left(\frac{\alpha q+c_0c'}{m_0}\right)
    \nonumber\\
    &\hspace{4cm} \times  \,    \int_{0}^{\infty} \ (\mathcal{H}_{k}g)\left(x/q^2\right)   J_{k-1}\left(\frac{4\pi\sqrt{\ell x}}{m_{0}m'}\right) 
e\left(\frac{c_{0}c'x}{qm_{0}m'}\right)  \ dx.
\end{align*}
Upon recombining the $c_{0}$-sum and $c'$-sum via $c=c_0c'$, it follows that
\begin{align}\label{SkBeforePoisson}
    S_k(\ell; \chi)  \ = \ & \  \frac{2\pi}{q} \frac{\epsilon_\chi}{\sqrt{q}}\, \sum_{m_0|q^\infty}\sum_{\substack{m'\geq1\\(m',q)=1}}\sum_{\substack{c\neq0\\ (c,m')=1}} \  \frac{\overline{\chi}(m')}{m_0m'}  e_{m'}\left(\ell q\overline{cm_0}\right)\sumast_{\substack{\alpha\, (m_0)\\ c\equiv-\alpha q\, (m_0)}}e_{m_0}\left(-\ell\overline{\alpha m'}\right)\chi\left(\frac{\alpha q+c}{m_0}\right)
    \nonumber\\
    &\hspace{4cm} \times \,  \int_{0}^{\infty} \ (\mathcal{H}_{k}g)\left(x/q^2\right)   J_{k-1}\left(\frac{4\pi\sqrt{\ell x}}{m_{0}m'}\right) 
e\left(\frac{cx}{qm_{0}m'}\right)  \ dx.
\end{align}
This completes the arithmetic preparation as described in Sect. \ref{applyhankcanc}.

\subsection{Step 6: Second Poisson}\label{reversPois}

With the preparations carried out above,  we rewrite the $c$-sum of (\ref{SkBeforePoisson}) as
\begin{align}\label{csumBeforePoisson}
    \sum_{\substack{c\neq0\\(c,m')=1}}\mathcal{D}_\chi(c;m_0,m') \int_0^\infty \ (\mathcal{H}_{k}g)\left(x/q^2\right) 
J_{k-1}\left(\frac{4\pi\sqrt{\ell x}}{m_0m'}\right)  e\left(\frac{cx}{m_0m'q}\right) \ dx \ = \ T_1 \ - \ T_2,
\end{align}
where \begin{align*}
    T_1 \, &:= \,  \sum_{\substack{c\in\Z}}\mathcal{D}_\chi(c;m_0,m') \, \int_0^\infty\ (\mathcal{H}_{k}g)\left(x/q^2\right)  J_{k-1}\left(\frac{4\pi\sqrt{\ell x}}{m_0m'}\right) e\left(\frac{cx}{m_0m'q}\right) dx,\\
    T_2 \, &:= \, \delta(m'=1)\mathcal{D}_\chi(0;m_0,1)  \, \int_0^\infty\ (\mathcal{H}_{k}g)\left(x/q^2\right) J_{k-1}\left(\frac{4\pi\sqrt{\ell x}}{m_0}\right)  dx,\\
    \mathcal{D}_\chi(c;m_0,m') \, &:= \, \delta((c,m')=1)\overline{\chi}(m')e_{m'}\left(\ell q\overline{cm_0}\right)\sumast_{\substack{\alpha\, (m_0)\\ c\equiv-\alpha q\, (m_0)}}e_{m_0}\left(-\ell\overline{\alpha m'}\right)\chi\left(\frac{\alpha q+c}{m_0}\right).
\end{align*}

The treatment of $T_2$ is simpler. Observe that 
\begin{align*}
    \mathcal{D}_\chi(0;m_0,1)  =  \sumast_{\substack{\alpha\, (m_0)\\ 0\equiv \alpha q\, (m_0)}}e_{m_0}\left(-\ell\overline{\alpha}\right)\chi\left(\frac{\alpha q}{m_0}\right) 
    = \ \delta(m_0=q)\sumast_{\alpha\,(q)}\,\, \overline{\chi}(\alpha)e\left(-\frac{\ell\alpha}{q}\right)
    =  \delta(m_0=q)\sqrt{q}\,\overline{\epsilon_\chi}\chi(\ell).
\end{align*}
A change of variables $x\to q^2 x$ and the Hankel inversion formula (Lem. \ref{hankelinv}) imply that 
\begin{align}\label{T2Eval}
    T_2 \, = \, \delta(m_0=q, m'=1)\, \overline{\epsilon_\chi}\chi(\ell) \, \frac{q^{5/2}g(\ell)}{2\pi}.
\end{align}

For $T_1$,  the character sum $\mathcal{D}_\chi$ can be expressed in terms of Kloosterman sums via \eqref{FirstPoissonCharComp2} and \eqref{FirstPoissonCharComp}:
\begin{align}
    \mathcal{D}_\chi(c;m_0,m')&=  \sumast_{\substack{x\,(m)\\ c\equiv -xq\, (m)}}e_{m}(-\ell \overline{x})\chi\left(\frac{qx+c}{m}\right) 
    =  \frac{\epsilon_\chi}{m\sqrt{q}}\sum_{\gamma\,(mq)}\overline{\chi}(\gamma)\sumast_{x\,(m)}e_m(-\gamma x-\ell\overline{x})e_{mq}(-c\gamma) \nonumber\\
    &= \frac{\epsilon_\chi}{m\sqrt{q}}\sum_{\gamma\,(mq)}\overline{\chi}(\gamma)S(\gamma,\ell;m)e_{mq}(-c\gamma). \label{anotherGauss}
\end{align}
Note: $m=m_0m'$. Plugging (\ref{anotherGauss}) into $T_{1}$ and making a change of variables $x\to mq x$, we arrive at
 \begin{align*}
    T_1=\epsilon_\chi \sqrt{q}\sum_{\substack{c\in\Z}}\sum_{\gamma\,(mq)}\overline{\chi}(\gamma)S(\gamma,\ell;m)e_{mq}(-c\gamma) 
    \int_0^\infty\, (\mathcal{H}_{k}g)\left(mx/q\right) J_{k-1}\bigg(4\pi \sqrt{\frac{\ell q x}{m}}\ \bigg)  e(cx)  dx. 
\end{align*}
We now readily observe the role reversal of the $c$-sum and $m$-sum as discussed in Sect. \ref{geodual}! 

The bounds for the Hankel transform and Bessel functions recorded in Sect. \ref{prelimsec} allow us to apply the dominated convergence theorem and obtain
\begin{align*}
    T_1=\epsilon_\chi \sqrt{q}\sum_{\gamma\,(mq)}\overline{\chi}(\gamma)S(\gamma,\ell;m) \lim_{\epsilon \to 0+} \, \sum_{\substack{c\in\Z}} e_{mq}(-c\gamma)
    \int_{\R}\,(\mathcal{H}_{k}g)\left(mx/q\right) J_{k-1}\bigg(4\pi \sqrt{\frac{\ell q x}{m}}\ \bigg)h_{\epsilon}(x)  e(cx) \ dx,
\end{align*}
where $h_{\epsilon}$ is a smooth function on $\R$ such that $h_{\epsilon}\equiv 1$ on $[\epsilon, \infty)$, $h_{\epsilon}\equiv 0$ on $(-\infty, 0]$, and $0\le h_{\epsilon}\le 1$ on $(0, \epsilon)$. Inside the limit, we apply Poisson summation $(\bmod\, mq)$ to the $c$-sum.  By dominated convergence again,
\begin{align}\label{T1Eval}
    T_1 \ = \ 2\pi\, \epsilon_\chi\sqrt{q}\sum_{c=1}^\infty\overline{\chi}(c)S(c,\ell;m) J_{k-1}\bigg(\frac{4\pi\sqrt{\ell c}}{m}\bigg) \int_{0}^{\infty}   g(y) J_{k-1}\left(\frac{4\pi\sqrt{cy}}{q}\right)dy
\end{align}
upon taking the limit $\epsilon\to 0+$.
Inserting \eqref{T2Eval}--(\ref{T1Eval}) into \eqref{SkBeforePoisson}--\eqref{csumBeforePoisson}, we readily observe that 
\begin{align}\label{SkAfter2ndPoisson}
    S_k(\ell; \chi)  \, = \, & \ 4\pi^2 \frac{\epsilon_\chi^2}{q}\sum_{m=1}^\infty\frac{1}{m}\sum_{c=1}^\infty \overline{\chi}(c)S(c,\ell;m) J_{k-1}\bigg(\frac{4\pi\sqrt{\ell c}}{m}\bigg) \int_{0}^{\infty} g(y) J_{k-1}\left(\frac{4\pi\sqrt{cy}}{q}\right)dy \, - \, g(\ell)\chi(\ell) 
\end{align}
by combining the $m_0$-sum and $m'$-sum.


\subsection{Step 7:  Petersson in reverse}

Inserting \eqref{SkAfter2ndPoisson} back into \eqref{dualPoisspl} yields \begin{align}
    I_{k}(\ell; \chi )  \, = & \,  2\pi i^{-k}\frac{\epsilon_\chi^2}{q}\overline{\chi}(\ell)	\int_{0}^{\infty}   g(y)J_{k-1}\bigg(\frac{4\pi\sqrt{y\ell}}{q}\bigg)  \, dy\nonumber\\
    &\ + \ 4\pi^2 \frac{\epsilon_\chi^2}{q}\sum_{m=1}^\infty\frac{1}{m}\sum_{c=1}^\infty \overline{\chi}(c)S(c,\ell;m) J_{k-1}\bigg(\frac{4\pi\sqrt{\ell c}}{m}\bigg)\int_{0}^{\infty}   g(y) J_{k-1}\left(\frac{4\pi\sqrt{cy}}{q}\right)dy.
\end{align}
Applying the Petersson formula \eqref{Peter} in reverse to the $m$-sum, Thm. \ref{ourmainthmtwist} follows.

\begin{remark}
Readers should note that the diagonal term $g(\ell)\chi(\ell)$ of (\ref{finiPet}) from the initial application of the Petersson trace formula conveniently cancelled with the dual zeroth frequency from the Poisson summation. Only after this does the crucial ``role-reversal'' of sums takes place. 
\end{remark}

\section{Proof of Corollary \ref{maincor}}\label{proofmainisola}

In App. \ref{AC&PGsec}, we show that the $L$-series (\ref{holoDS}) can be analytically continued to $\re s> -(k-6)/2$, and 
\begin{align}\label{avgACcont}
\mathcal{A}_{\ell}(s, \chi) \ := \  \sideset{}{^h}{\sum}_{f \in \mathcal{B}_{k}(1)} \ \lambda_{f}(\ell) L(s,f\times \chi)
\end{align}
has \emph{polynomial growth} on the vertical strip $0< \re s<1$ and as $|\im s|\to \infty$. Granting this (i.e., Prop. \ref{prelimbd}), we prove Cor. \ref{maincor} as follows. Pick any $g\in C_{c}^{\infty}(0,\infty)$. Then its Mellin transform $\mathcal{G}(s)$ is entire and it follows from repeated integration by parts that  for any $A>0$, 
\begin{align}\label{Gprop}
    |\mathcal{G}(s)| \ \ll_{A, \, \re s} \  (1+ |\im s|)^{-A}.
\end{align}

Let $\sigma \in (0,1)$, $\ell\geq1$ and $k\ge 6$. In (\ref{truncIsum}), apply (\ref{melT}) and rearrange sums and integrals, we have
\begin{align}
 I_k(\ell; \chi)  
\ &= \   \, \int_{(3/2)} \, \mathcal{G}(s) \ \sideset{}{^h}{\sum}_{f \in \mathcal{B}_{k}(1)} \, \lambda_{f}(\ell) L(s, f\times \chi) \, \frac{ds}{2\pi i}. \nonumber
\end{align} 
By Prop. \ref{prelimbd} and (\ref{Gprop}),  we may shift the line of integration above to $\re s =\sigma$. 

Next, it follows from (\ref{bessgamm}) and  (\ref{melT}) that
\begin{align}\label{Hankmel}
    \int_{0}^{\infty} \,  g(y)J_{k-1}\left(\frac{4\pi\sqrt{cy}}{q}\right)  dy \, = \,  \frac{1}{2\pi}\int_{(3/2)}\mathcal{G}(s)\frac{\gamma_k(1-s)}{\gamma_k(s)}\left(\frac{c}{q^2}\right)^{s-1}\frac{ds}{2\pi i}.
\end{align}
Using (\ref{Gprop}), (\ref{ratiostir}) and the holomorphy of $\mathcal{G}(s)$, we may shift the line of integration to $\re s = -1/2$ in (\ref{Hankmel}). Inserting the resultant into the right-hand side of (\ref{mainglobid}), we deduce that
\begin{align*}
    I_k(\ell; \chi) \, = \, i^k \epsilon_\chi^2 \,\sum_{n=1}^\infty \la_f(n)\overline{\chi}(n) \, \sideset{}{^h}{\sum}_{f \in \mathcal{B}_{k}(1)} \ \lambda_{f}(\ell)\int_{(-1/2)} \ \mathcal{G}(s)   \,  \frac{\gamma_{k}(1-s)}{\gamma_{k}(s)} \left(\frac{n}{q^2}\right)^{s-1} \frac{ds}{2\pi i}.
\end{align*}
Upon exchanging the order of sums and integrals, we find that the $c$-sum converges absolutely and is precisely the Dirichlet series $L(1-s, f\times \overline{\chi})$. In other words,  
\begin{align}\label{secsum}
    I_k(\ell; \chi)= \ i^k \epsilon_\chi^2\int_{(-1/2)} \ \mathcal{G}(s)q^{1-2s}  \frac{\gamma_{k}(1-s)}{\gamma_{k}(s)}  \sideset{}{^h}{\sum}_{f \in \mathcal{B}_{k}(1)} \ \lambda_{f}(\ell) L(1-s, f\times \overline{\chi}) \, \frac{ds}{2\pi i}.
\end{align}
By Prop. \ref{prelimbd}, (\ref{Gprop}) and (\ref{ratiostir}), we may shift the line of integration for (\ref{secsum}) back to $\re s= \sigma$.  As a result, we have
\begin{align}
        \int_{(\sigma)} \, \mathcal{G}(s)   \left\{  \ \ 	\sideset{}{^h}{\sum}_{f \in \mathcal{B}_{k}(1)} \ \lambda_{f}(\ell)\left( L(s, f\times\chi)  \, -   \, i^k \epsilon_\chi^2q^{1-2s}\,  \frac{\gamma_{k}(1-s)}{\gamma_{k}(s)}    \,  L(1-s, f\times \overline{\chi}) \right) \right\} \,  \frac{ds}{2\pi i} \, = \,  0.
    \end{align}



Let $\mathcal{B}_{k}(1):=\left\{ f_{1}, \ldots, f_{d}\right\}$. \footnote{ Here, we use the finite dimensionality of the linear space of holomorphic cusp forms of a given weight and level. } The vectors $ \left( \lambda_{f_{i}}(1), \,\lambda_{f_{i}}(2),  \, \ldots \,  \right)$ for $i=1,\ldots, d$ are linearly independent over $\C$, and thus, there exists $\ell_{1}< \cdots < \ell_{d}$ such that the submatrix $A:= \left(\lambda_{f_{i}}(\ell_{j})\right)_{\substack{1\le i, j\le d}}$ is invertible. By Prop. \ref{prelimbd},  the function $s  \mapsto  \mathcal{A}_{\ell_{j}}(s, \chi)$
admits a holomorphic continuation (say $G_{\ell_{j}}(s, \chi)$)  to the region  $\re s>-(k-6)/2$ for $j=1,\ldots, d$:
\begin{align}\label{ACvecto}
    \left( \frac{L(s,f_{1}\times \chi)}{||f_{1}||^2}, \ \ldots, \ \frac{L(s,f_{d}\times \chi)}{||f_{d}||^2} \right)  \ =  \  \frac{(4\pi)^{k-1}}{\Gamma(k-1)}\left(G_{\ell_{1}}(s, \chi), \, \ldots, \, G_{\ell_{d}}(s, \chi)\right)A^{-1}.
\end{align}
Hence, each of $L(s, f_{i}\times \chi)$ admits a holomorphic continuation to the same region.

Since (\ref{mainglobid}) holds for any $\sigma\in (0,1)$ and any  $g\in C^{\infty}_{c}(0, \infty)$, we have, on the vertical strip $0<\re s <1$, 
\begin{align*}
     \left( \frac{L(s, f_{i} \times \chi)}{||f_{i}||^2} - i^k \epsilon_\chi^2q^{1-2s}  \frac{\gamma_{k}(1-s)}{\gamma_{k}(s)} \, \frac{L\left(1-s, f_{i}\times \overline{\chi}\right)}{||f_{i}||^2}\right)_{1\le i\le d} \, \cdot \,  A \ = \  (0,\ldots, 0). 
\end{align*}
Since $A$ is invertible, the functional equation (\ref{stdfunceq})  for $f=f_{i}$ on $0<\re s <1$ follows immediately. 


From the holomorphic continuation of $L(s, f_{i} \times \chi)$ to $\re s> -(k-6)/2$ and the functional equation just proved, $L(s, f_{i} \times \chi)$ also extends holomorphically to $\re s< k/2-2$ . As a result, $L(s, f_{i}\times \chi)$ admits an entire continuation and now the functional equation holds for all $s\in \C$. This completes the proof. 

\appendix

\section{Analytic continuation and polynomial growth}\label{AC&PGsec}

The proof of Prop. \ref{prelimbd} below is similar to that in \cite{KLev},  but requires several adjustments to accommodate our current setting. For completeness and convenience of the reader, we supply a short argument as follows.   
 
\begin{prop}\label{prelimbd}
Let  $\ell\ge 1$ and $k\ge 6$ be integers.   Then the function $\mathcal{A}_{\ell}(s, \chi)$  defined in (\ref{avgACcont}) admits a holomorphic continuation to the half-plane \,$\re s> -(k-6)/2$ and satisfies the estimate 
    \begin{align}\label{preast}
    \mathcal{A}_{\ell}(s, \chi) \  \ll \ (1+|t|)^{k-3}
\end{align}
for any $s=\sigma+it$ with\,  $\sigma > -(k-6)/2 $ and $t\in \R$. The implicit constant depends only on $k, \ell, q, \sigma$. 

\end{prop}

\begin{proof}
Take $g\in C_{c}^{\infty}[1,2]$ such that $1 =   \sum_{u\in \Z} \ g(x/2^u)$ for any $x>0$. Inserting this into (\ref{holoDS}) for $\re s\gg 1$,
observe that (\ref{avgACcont}) can be written as
\begin{align}
 \mathcal{A}_{\ell}(s,\chi)   =    \sum_{u=-1}^{\infty} \frac{\mathcal{I}_{s}(2^u;\ell,\chi)}{2^{us}}, \hspace{10pt}\text{ where } \hspace{10pt}   \mathcal{I}_{s}(X;\ell, \chi)   :=  \sideset{}{^h}{\sum}_{f \in \mathcal{B}_{k}(1)} \ \lambda_{f}(\ell)  \sum_{n} \ \lambda_{f}(n)\chi(n)G_{s}(n/X),
\end{align} 
and  $ G_{s}(y):=  y^{-s} g(y)$.  For  $X> 2q^2\ell$, we have $G_{s}(\ell/X)=0$, and from  (\ref{dualPoisspl})--(\ref{SkDef}), it follows that
\begin{align*}
    \frac{1}{X} \, \mathcal{I}_{s}(X;\ell, \chi) \ll   \bigg|\int_{0}^{\infty} \, G_{s}(y)J_{k-1}\bigg(\frac{4\pi\sqrt{yX\ell}}{q}\bigg)  dy \bigg|  + \sum_{c\ge 1} \sum_{\substack{m\neq0 }} \  \bigg|\int_{0}^{\infty} G_{s}(y)  J_{k-1}\bigg(\frac{4\pi\sqrt{yX\ell}}{c}\bigg) e\bigg(-\frac{myX}{cq}\bigg) dy\bigg|.
\end{align*}
Bounding the second summand on the right-hand side is harder, and this will be our focus. 

We split the $c$-sum above into two parts, according to the conditions  $c>\sqrt{\ell X}/20$ and  $c\le \sqrt{\ell X}/20$.  These two parts are denoted by $ \mathcal{I}^{(2)}_{s, >}(X;\ell, \chi)$ and $\mathcal{I}^{(2)}_{s, \le}(X; \ell, \chi)$ respectively. 

By Lem. \ref{Melexpthm}, we have
\begin{align}
    \int_{\R} G_{s}(y)  J_{k-1}\bigg(\frac{4\pi\sqrt{yX\ell}}{c}\bigg)  e\bigg(-\frac{myX}{cq}\bigg)  dy =      \frac{1}{2\pi}    \int_{(a)}  \frac{\gamma_{k}(1-v)}{\gamma_{k}(v)}   \bigg(\frac{\sqrt{X\ell }}{c}\bigg)^{2(v-1)}\int_{\R}  G_{s}(y)    y^{v-1}  e\bigg(-\frac{myX}{cq}\bigg)  dy \frac{dv}{2\pi i}, \nonumber
\end{align}
where $1< a<(k+1)/2$.  Then integrating by parts $r\ge 2$ times, it follows that the above expression is
\begin{align}
  \ \ll \  \left(\frac{\sqrt{\ell X}}{c}\right)^{2(a-1)}  \bigg(\frac{|m|X}{cq}\bigg)^{-r} \, \int_{(a)} \ \left|\frac{\gamma_{k}(1-v)}{\gamma_{k}(v)}  \right| \, \int_{\R} \ \left|\frac{d^r}{dy^r} \left[G_{s}(y)y^{v-1}\right]\right| \,  dy \  |dv|. \nonumber
\end{align}
Let $v= a +i\tau$ and $s= \sigma+it$ with $\sigma, t, \tau \in \R$. Then Stirling's formula gives
\begin{align}\label{ratiostir}
    \left|\frac{\gamma_{k}(1-v)}{\gamma_{k}(v)}  \right| \ \asymp_{k, a} \ (1+ |\tau|)^{1-2a}. 
\end{align}
For any $r\ge 2$ and $y\in [1,2]$,  we have 
\begin{align}
 \max_{y\in \R}  \,  \left|G_{s}^{(r)}(y)\right| \ &\ll_{r} \ \max_{\substack{i+j=r \\ i,j\ge 0 }} \ |g^{(i)}(y)| \, \cdot \, | s(s+1)\cdots (s+j-1)| y^{-\sigma-j}   \ \ll_{r, \sigma} \ \left(1 + |t|\right)^r, \nonumber\\
    \left|\frac{d^r}{dy^r} [G_{s}(y)y^{v-1}]\right| &\ \ll_{r} \ \max_{\substack{i+j=r\\ i,j\ge 0 }} \ |G_{s}^{(i)}(y)| \cdot | (v-1)\cdots (v-j) y^{v-j-1} | 
    \ \ll_{a, r, \sigma} \ \left(\left(1 + |t|\right) \left(1 + |\tau|\right)\right)^r. 
\end{align}
Hence, if  $1< a<(k+1)/2$ and $r\in \Z$ such that $2\le r< 2a-3$, then 
\begin{align}\label{largecest}
    \bigg| \int_{\R} \, G_{s}(y)  J_{k-1}\bigg(\frac{4\pi\sqrt{yX\ell }}{c}\bigg)  e\bigg(-\frac{myX}{cq}\bigg) \, dy\bigg| 
    \ \ll  \    \left(1+ |t|\right)^r \frac{X^{a-r-1}}{|m|^r c^{2a-2-r}},
\end{align}
\begin{align}\label{smaest}
    \mathcal{I}^{(2)}_{s, >}(X;\ell, \chi) \ \ll_{\ell, q, k, \sigma, r, a}  \ \left(1+ |t|\right)^r \frac{X^{a-r-1}}{(\sqrt{X})^{2a-3-r}} \ =  \ (1+|t|)^r X^{(1-r)/2}.
\end{align}

Next, we estimate $\mathcal{I}^{(2)}_{s, \le}(X; \ell, \chi)$ with  $X> 2q^2\ell$. By (\ref{asympJbess}),  it suffices to estimate the oscillatory integral:  
 \begin{align}
     \int_{\R} \, G_{s}(y)  W_{k}\bigg(\frac{2\sqrt{y\ell X}}{c}\bigg) e\bigg(\frac{2q\sqrt{y\ell X}-mXy}{cq}\bigg)  dy. \nonumber
 \end{align}
To this end, we make use of  Lem. \ref{BKYibp} with the functions
 \begin{align*}
     w_{s}(y) \, := \, G_{s}(y)  W_{k}\bigg(\frac{2\sqrt{y\ell X}}{c}\bigg)  \hspace{20pt} \text{ and } \hspace{20pt} h(y) \, := \, \frac{2q\sqrt{y\ell X}-mXy}{cq}. 
 \end{align*}
Suppose  $y\in [1, 2]$ and $r\ge 2$. We observe the following bounds:
\begin{align*}
    h'(y) = \frac{qy^{-1/2}\sqrt{\ell X} -mX}{cq},  \hspace{10pt} |h'(y)|\gg \frac{|m|X}{cq} \bigg(1- \frac{1}{|m|} \sqrt{\frac{q^2\ell}{X}}\bigg) \gg  \frac{|m|X}{cq},  \hspace{10pt} 
    |h^{(r)}(y)|  \asymp_{r}  \frac{\sqrt{\ell X}}{c}   \ge  20,
\end{align*}
 \begin{align}\label{weightfbd}
     \left|w_{s}^{(r)}(y)\right| 
     \ &\ll_{r,\sigma} \,  \max_{i+j=r} \,  \big| G_{s}^{(i)}(y)\big|  \max_{\substack{m_{1}+ 2m_{2} +\cdots +j\cdot m_{j}=j \\ m_{1},\ldots, m_{j}\ge 0}} \, \big|W_{k}^{(m_{1}+\cdots + m_{j})}\big|\bigg(\frac{2\sqrt{y\ell X}}{c}\bigg) \, \cdot \, \prod_{i=1}^{j}   \bigg| \partial_{y}^i \bigg(\frac{2\sqrt{y\ell X}}{c}\bigg) \bigg|^{m_{i}} \nonumber\\
     \ &\ll_{r,k} \, (1 + |t|)^r \bigg(\frac{2\sqrt{\ell X}}{c}\bigg)^{-1/2}, 
 \end{align}
  where Leibniz's rule and  Fa\`{a} di Bruno's formula were used. Apply Lem. \ref{BKYibp} with the parameters:
\begin{align}
    W \ = \  \bigg(\frac{2\sqrt{\ell X}}{c}\bigg)^{-1/2}, \hspace{10pt} V \ = \  \left(1+|t|\right)^{-1}, \hspace{10pt} H \ = \  \frac{\sqrt{\ell X}}{c},  \hspace{10pt} G \ = \  1, \hspace{10pt} R \ = \ \frac{|m|X}{cq},
\end{align}
we conclude that: 
\begin{align}\label{smallc}
    \bigg| \int_{\R} \, w_{s}(y) e\left(h(y)\right) \, dy \bigg| \ &\ll \   \,\bigg(\frac{2\sqrt{\ell X}}{c}\bigg)^{-1/2} \left(1+|t|\right)^{A} \bigg( \, \frac{cq}{|m|X} \ + \ \frac{\sqrt{\ell X}}{c} \bigg(\frac{cq}{|m|X}\bigg)^2 \,  \bigg)^{A} \nonumber\\
    \ &\ll \   \, \bigg(\frac{\sqrt{\ell X}}{c}\bigg)^{-1/2} \left(\frac{cq\left(1+|t|\right)}{|m|X}\right)^{A} \bigg(1 \ + \ \frac{1}{|m|}\sqrt{\frac{q^2\ell}{X}}  \bigg)^{A}
    \ \ll \ \frac{\sqrt{c}}{X^{1/4}}  \left(\frac{c\left(1+|t|\right)}{|m|X}\right)^{A}
\end{align} 
for $c>0$, $m\neq 0$, $X>2q^2\ell$, where the implicit constants depend on $k,\ell, q, A,\sigma$.  When $A>1$, we have
 \begin{align}\label{larest}
    \mathcal{I}^{(2)}_{s, \le}(X; \ell, \chi) \ \ll \ X^{-1/4}  \, \bigg(\frac{\left(1+|t|\right)}{X}\bigg)^{A} \ \sum_{c\ll \sqrt{X}}  \ c^{A+1/2} \ \ll \ \sqrt{X} \, \bigg(\frac{\left(1+|t|\right)}{\sqrt{X}}\bigg)^{A}. 
\end{align}

As a result, we obtain the estimate:
\begin{align*}
     \frac{1}{X} \, \mathcal{I}_{s}(X;\ell, \chi) \ \ll \ (1+|t|)^r X^{(1-r)/2} \ + \ \sqrt{X} \, \bigg(\frac{\left(1+|t|\right)}{\sqrt{X}}\bigg)^{A},
\end{align*}
provided that $1< a<(k+1)/2$, $r\in \Z$ with $2\le r< 2a-3$, $X>2q^2\ell$,  and $A>1$. By taking $r=k-3$, $a= (k+1)/2-\epsilon$, $A=k-3$ with $k\ge 6$, we have
\begin{align}
     \mathcal{I}_{s}(X;\ell, \chi) \ \ll_{k, \ell, q, \sigma} \ \left(1+|t|\right)^{k-3} X^{-(k-6)/2}. \label{larbd} 
\end{align}
On the strip $\sigma_{1} \le \sigma\le \sigma_{2}$, the same estimates hold with the implicit constants depend only on $k, \ell, q, \sigma_{1}, \sigma_{2}$. 


 We now turn to $\mathcal{A}_{\ell}(s,\chi)$. For $\re s\gg 1$, we have
    \begin{align}\label{dyaspl}
  \mathcal{A}_{\ell}(s,\chi)  =    \bigg(\, \sum_{\substack{2^u> 2\ell \\ u\ge -1} } \, + \, \sum_{\substack{2^u\le  2\ell \\ u\ge -1} } \, \bigg) \  \frac{\mathcal{I}_{s}(2^u;\ell,\chi)}{2^{us}} 
\end{align}
Apply (\ref{larbd})  to the first sum of (\ref{dyaspl}), we have, for $\sigma > -(k-6)/2$ and $t\in \R$,
\begin{align}\label{largell}
    \sum_{\substack{2^u> 2\ell \\ u\ge -1} } \ \bigg|\frac{\mathcal{I}_{s}(2^u;\ell, \chi)}{2^{us}}\bigg| \ \ll \  \left(1+|t|\right)^{k-3} \sum_{\substack{u\ge 0} } \frac{ (2^u)^{-\frac{k-6}{2}}}{2^{u\sigma}} \ \ll_{k, \ell, q, \sigma} \ \left(1+|t|\right)^{k-3}.
\end{align}
Next, it follows from Cauchy--Schwarz's inequality and bounding the geometric side of (\ref{Peter}) that
\begin{align}\label{smallell}
   \hspace{-5pt} \sum_{\substack{2^u\le  2\ell \\ u\ge -1} } \  \bigg|\frac{\mathcal{I}_{s}(2^u;\ell, \chi)}{2^{us}} \bigg|  \ll     \sum_{n\le 4\ell} \ \  \sideset{}{^h}{\sum}_{f \in \mathcal{B}_{k}(1)} \  \ \left|\lambda_{f}(\ell)\lambda_{f}(n)\right|n^{-\sigma} 
     \ll  \sum_{n\le 4\ell} \ \  \sideset{}{^h}{\sum}_{f \in \mathcal{B}_{k}(1)} \ \left(|\lambda_{f}(\ell)|^2 + |\lambda_{f}(n)|^2\right)  \ll_{k,\ell, q}  1.
\end{align}

From (\ref{largell}) and (\ref{smallell}), the series on the right side of (\ref{dyaspl}) converge absolutely pointwisely and uniformly on every compact subset of the region $\re s >-(k-6)/2$. Thus, $\mathcal{A}_{\ell}(s,\chi)$ admits an analytic continuation and (\ref{preast}) holds in the same region.  
\end{proof}

\printbibliography

\end{document}